\documentclass[11pt,reqno]{amsart}
\usepackage{mathrsfs}
\usepackage{cite}

\def\subjclass#1{{\renewcommand{\thefootnote}{}%
\footnote{\emph{Mathematics Subject Classification (2020):} #1}}}

\usepackage{amsmath}
\usepackage{amssymb}
\usepackage{amsfonts}

\DeclareMathOperator{\divg}{div}

\date{\today}

\hoffset=- 2cm \voffset=- 0cm 
 \theoremstyle{plain}
\newtheorem{Thm}{Theorem}[section]
\newtheorem{Lem}{Lemma}[section]
\newtheorem{Assum}{Assumption}[section]

\newtheorem{Rem}{Remark}[section]
\newtheorem{Cor}{Corollary}[section]

\usepackage{amssymb}
\usepackage{color}

\def\0{\mathbf 0}

\def\v{\vskip}

\errorcontextlines=0 \numberwithin{equation}{section}
\numberwithin{Thm}{section}
\allowdisplaybreaks

\begin{document}
\large
%Topmatter

\title[Liouville-type theorems]
{Two improved Liouville-type theorems for the 3D stationary tropical climate model without temperature assumptions}

\author{Zhibing Zhang}

\address{Zhibing Zhang: School of Mathematics and Statistics, Key Laboratory of Modeling, Simulation and Control of Complex Ecosystem in Dabie Mountains of Anhui Higher Education Institutes, Anqing Normal University, Anqing 246133, China}
\email{zhibingzhang29@126.com and zhibingzhang29@aqnu.edu.cn}%

%zhibingzhang29@aqnu.edu.cn;
\thanks{}



\begin{abstract}
In this paper, we establish two improved Liouville-type theorems for the three-dimensional stationary tropical climate model without imposing any temperature assumptions.
Compared with previous results, we achieve significant improvements on the growth assumptions imposed on $u$, $v$ and $\nabla v$. Owing to the special structure inherent to the equations, we fully exploit the $L^2$-estimate for the oscillation of the temperature, refined interpolation techniques and an iteration argument to establish a Liouville-type theorem under flexible and relaxed power-law growth conditions. As an interesting consequence, we obtain the triviality of solutions under the integrability conditions imposed only on $(u,\nabla v)$ or $(\nabla u,\nabla v)$. Furthermore, by developing a systematic framework to handle the energy function associated with the non-trivial solutions and with the aid of delicate ODE analysis and a contradiction argument, we refine our Liouville-type theorem by introducing logarithmic corrections to the growth rates.
\end{abstract}
\maketitle
%end topmatter

%\tableofcontents

\section{Introduction}
We consider the following stationary tropical climate model in $\mathbb{R}^3$:
\begin{equation}\label{equ1.1}
  \left\{
    \begin{array}{ll}
     -\Delta u+(u\cdot\nabla) u+\nabla \pi+\divg(v\otimes v)=0,  \\
   -\Delta v+(u\cdot\nabla) v+\nabla \theta+(v\cdot\nabla)u=0, \\
   -\Delta \theta+u\cdot\nabla\theta+\divg v=0,\\
     \divg u=0,
\end{array}
  \right.
\end{equation}
where $u= (u_1 , u_2 , u_3 )$ is the barotropic mode, $v= (v_1 , v_2 , v_3 )$ is the first baroclinic mode of vector velocity, and $\theta$ and $\pi$ represent the temperature and pressure, respectively. The tropical climate model was originally proposed by Frierson, Majda and Pauluis in \cite{FMP04} to characterize the large-scale dynamical behaviors of precipitation fronts within the tropical atmosphere. Comprehensive fundamental background concerning the tropical climate model is available in \cite{LT16,Majda03,MB03,CNTW26, DWWYZ19, LYZ26, IKK26}. From a mathematical perspective, \eqref{equ1.1} possesses a characteristic structural feature originating from the coupling of the divergence-free velocity field $u$ and the non-divergence-free velocity field $v$. This hybrid coupled structure induces substantial mathematical difficulties and nontrivial analytical challenges for theoretical investigation.

When $v=0$ and $\theta$ is a constant, \eqref{equ1.1} is reduced to the stationary incompressible Navier-Stokes equations:
\begin{align}\label{equ1.2}
  \left\{
    \begin{array}{ll}
     -\Delta u+(u\cdot\nabla) u+\nabla \pi=0,  \\
     \divg u=0.
\end{array}
  \right.
\end{align}
A challenging Liouville-type problem for \eqref{equ1.2} is whether \eqref{equ1.2} has non-trivial solutions with the finite Dirichlet energy
\begin{align*}
\int_{\mathbb{R}^{3}}|\nabla u|^{2}dx<+\infty,
 \end{align*}
and the vanishing condition at infinity
\begin{align*}
\lim_{|x|\rightarrow+\infty}u(x)=0.
\end{align*}
This Liouville-type problem was explicitly proposed in Galdi's book \cite[Remark X.9.4]{Galdi}.
Since the original problem is too difficult, various extra or alternative assumptions have been introduced to obtain Liouville-type theorems for \eqref{equ1.2}, see \cite{BGWX25,BY25,CPZ20,CPZZ20,Chae14,Chae25,Chae26,CW16,CW19,CJL21,CNY24,CY26a,CoY26,KNSS09,KTW17,Seregin16,Seregin18,SW19,Tsai21} and the references therein.

In recent years, Liouville-type results for the tropical climate model have attracted increasing attention.
To state Liouville-type results conveniently, we introduce some notations.
Let $A_R$ represent the annulus $B_{2R}\backslash \overline{B_R}$. We denote
\begin{align}
&X_{p,\alpha}(R)=\left(\int_{A_R}\left(\frac{|u|}{|x|^\alpha}\right)^pdx\right)^\frac{1}{p},\;Y_{q,\beta}(R)=\left(\int_{A_R}\left(\frac{|v|}{|x|^\beta}\right)^qdx\right)^\frac{1}{q},\notag \\
&Z_{r,\gamma}(R)=\left(\int_{A_R}\left(\frac{|\theta|}{|x|^\gamma}\right)^rdx\right)^\frac{1}{r},\;Y'_{2,\gamma}(R)=\left(\int_{A_R}\left(\frac{|\nabla v|}{|x|^\gamma}\right)^2dx\right)^\frac{1}{2},\label{nota}\\
&X_{p,\alpha,\lambda}(R)=\left(\int_{A_R}\left(\frac{|u|}{|x|^\alpha\ln^\lambda |x|}\right)^pdx\right)^\frac{1}{p},\;Y_{q,\beta,\mu}(R)=\left(\int_{A_R}\left(\frac{|v|}{|x|^\beta\ln^\mu |x|}\right)^qdx\right)^\frac{1}{q}.\notag
\end{align}
Ding and Wu \cite{FW21} investigated the tropical climate model and proved a Liouville-type theorem under one of the following assumptions
$$
\aligned
&(\mathrm{i})~~~~u\in L^p\left(\mathbb{R}^{3}\right) ~~~~~~~~~~~~\text{and}~~~~~~~~~~~~v,\theta\in L^p\left(\mathbb{R}^{3}\right)\cap L^3\left(\mathbb{R}^{3}\right) ~~~~~~~~~~~~\text{with}~~~~~~~~~~~~ 3\leq p\leq \frac{9}{2},\\
&(\mathrm{ii})~~~~u\in L^{3}\left(\mathbb{R}^{3}\right), v \in L^{2}\left(\mathbb{R}^3\right)~~~~~~~~~~~~\text{and}~~~~~~~~~~~~\nabla u,\nabla v,\nabla\theta \in L^{2}\left(\mathbb{R}^3\right).
\endaligned
$$
Recently, Cho et al. \cite{CIY24} established an improved Liouville-type theorem for \eqref{equ1.1} under one of the following conditions
\begin{align*}
&(\mathrm{i})\;\liminf_{R\rightarrow+\infty}\left[X_{p,\alpha}(R)+Y_{q,\beta}(R)+Z_{r,\gamma}(R)\right]<+\infty, \text{ with }p\in\left(\frac{3}{2},3\right),\;q,r\in[1,2),\\
&(\mathrm{ii})\;\liminf_{R\rightarrow+\infty}\left[X_{p,\alpha}(R)+Y_{q,\beta}(R)\right]=0,\;\limsup_{R\rightarrow+\infty}Z_{r,\gamma}(R)<+\infty,\text{ with }p=3,\;q=r=2,\\
&(\mathrm{iii})\;\liminf_{R\rightarrow+\infty}\left[X_{p,\alpha}(R)+Z_{r,\gamma}(R)\right]=0,\;\limsup_{R\rightarrow+\infty}Y_{q,\beta}(R)<+\infty,\text{ with }p=3,\;q=r=2,
\end{align*}
where $\alpha$, $\beta$ and $\gamma$ are given respectively by
$$
\alpha=\frac{2}{p}-\frac{1}{3},\;\beta=\frac{3}{2q}-\frac{1}{4},\;\gamma=\frac{3}{2r}-\frac{1}{4}.
$$
In \cite{DFZ26}, Dong, Fang and the author of this paper improved and extended the above result.
Very recently, Cho and Yang \cite{CY26} obtained  Liouville-type theorems without temperature assumptions.
Let $q\in[1,2]$ and the parameters $\alpha$, $\beta$ and $\gamma$ be redefined as
\begin{align}\label{abc}
\alpha=\frac{2}{p}-\frac{1}{3},\;\beta=\frac{2}{q}-\frac{1}{3},\;\gamma=0.
\end{align}
They showed the triviality of the solution if $u$ and $v$ satisfy one of the following conditions
\begin{align}
&\mathrm{(A1)}\;\liminf\limits_{R\rightarrow+\infty}\left[X_{p,\alpha}(R)+Y_{q,\beta}(R)+Y'_{2,\gamma}(R)\right]<+\infty, \text{ with }p\in\left(\frac{3}{2},3\right),\label{c1.3}\\
&\mathrm{(A2)}\;\liminf\limits_{R\rightarrow+\infty}X_{p,\alpha}(R)=0,\;\limsup\limits_{R\rightarrow+\infty}\left[Y_{q,\beta}(R)+Y'_{2,\gamma}(R)\right]<+\infty, \text{ with }p=3.\label{c1.4}
\end{align}

Our aim of this paper is to extend and improve Cho-Yang's Liouville-type theorems without temperature assumptions.
We make full use of the $L^2$-estimate for the oscillation of $\theta$ and sharpened interpolation techniques, combined with the energy method and an iteration argument to
establish new Liouville-type theorems valid for much more flexible and relaxed power-law growth conditions. As an instructive consequence, we derive Liouville-type theorems merely requiring integrability constraints on either the pair \((u,\nabla v)\) or \((\nabla u,\nabla v)\).
Furthermore, we construct a systematic analytical framework to analyse the energy functional corresponding to nontrivial solutions and adopt proof by contradiction to further weaken the aforementioned power-law growth conditions via logarithmic corrections, leading to additional strengthened results.

Before stating our Liouville-type results, we need to introduce some basic assumptions on the parameters $p$, $q$, $\alpha$, $\beta$, $\gamma$, $\lambda$, $\mu$.
\begin{Assum}\label{a1.1}
Let $(p,q)\in\left(\frac{3}{2},3\right]\times[1,2]$ and $\alpha\in\left[0,\frac{2}{p}-\frac{1}{3}\right]$, $\beta\in\left[0,\frac{3}{q}-\frac{1}{2}\right]$, $\gamma\in\left[0,\gamma_{p,\alpha}\right)$, where
\begin{equation*}
\gamma_{p,\alpha}=
\begin{cases}
\min\left\{
\frac{6-p-p\alpha}{6-2p},
\frac{p\bigl[6-p-(p+3)\alpha\bigr]}
{(6-2p)(2p-3)}
\right\}, & \text{ if }\;\frac{3}{2}<p<3, \\
+\infty, & \text{ if }\;p=3.
\end{cases}
\end{equation*}
\end{Assum}

\begin{Assum}\label{a1.2}
\begin{itemize}
\item[(i)] When $\frac{6-2p}{6-p} + \frac{6-3q}{6-q} < 1$, we require
\begin{align}\label{ine1.3}
\frac{p}{6-p}\alpha+\frac{2q}{6-q}\beta\leq1.
\end{align}
\item[(ii)] When $\frac{6-2p}{6-p} + \frac{6-3q}{6-q} \geq 1$ and $\gamma=0$, we require that \eqref{ine1.3} holds.
\item[(iii)] When $\frac{6-2p}{6-p} + \frac{6-3q}{6-q} \geq 1$ and $\gamma>0$, we require
 \begin{align}\label{ine1.6}
 \frac{p}{6-p}\alpha+\frac{2q}{6-q}\beta+\left(\frac{6-2p}{6-p}+\frac{6-3q}{6-q}-1\right)\gamma<1.
 \end{align}
\end{itemize}
\end{Assum}

\begin{Assum}\label{a1.3}
Let $\lambda\in\left[0,\frac{3}{p}-1\right]$ and $\mu\geq0$.
\end{Assum}

\begin{Assum}\label{a1.4}
\begin{itemize}
\item[(i)]
When $\frac{6-2p}{6-p} + \frac{6-3q}{6-q} < 1$ and $\frac{p}{6-p}\alpha+\frac{2q}{6-q}\beta=1$, we require
$$
\frac{p}{6-p}\lambda+\frac{2q}{6-q}\mu\leq \frac{3-p}{6-p}+\frac{6-2q}{6-q}.
$$	
\item[(ii)] When $\frac{6-2p}{6-p} + \frac{6-3q}{6-q} \geq 1$ and $\gamma=0$,
we require
$$
\frac{p}{6-p}\lambda+\frac{2q}{6-q}\mu<1.
$$	
\item[(iii)] When $\beta=\frac{3}{q}-\frac{1}{2}$, we require
$$
p\neq3\;\;\text{ and }\;\;\frac{p}{6-p}\lambda+2\mu\leq \frac{3-p}{6-p}.
$$
\end{itemize}	
\end{Assum}

Our first Liouville-type result reads as follows.
\begin{Thm}\label{main1}
Let $(u,\pi,v,\theta)$ be a smooth solution of \eqref{equ1.1}. Assume that \eqref{c1.3} or  \eqref{c1.4} holds, where the parameters $p$, $q$, $\alpha$, $\beta$, $\gamma$ satisfy Assumptions \ref{a1.1} and \ref{a1.2}. Then $u=v=0$ and $\theta$ is a constant.
\end{Thm}

Roughly speaking, we can improve Theorem \ref{main1} by weakening the growth conditions to accommodate logarithmic factors.
For \(\gamma>0\), any fixed logarithmic factor in the growth bound for \(\nabla v\) can be absorbed into a sufficiently small increase of \(\gamma\), while preserving Assumptions 1.1 and 1.2. We therefore do not introduce a separate logarithmic exponent for \(\nabla v\) in this regime.
Note that the notations $X_{p,\alpha,\lambda}(R)$ and $Y_{q,\beta,\mu}(R)$ used in Theorem \ref{main2} are also defined in \eqref{nota}.

\begin{Thm}\label{main2}
Let $(u,\pi,v,\theta)$ be a smooth solution of \eqref{equ1.1}. Suppose the parameters $p$, $q$, $\alpha$, $\beta$, $\gamma$, $\lambda$, $\mu$ satisfy Assumptions \ref{a1.1}, \ref{a1.2}, \ref{a1.3} and \ref{a1.4}.
Furthermore, assume that one of the following conditions holds
\begin{align*}
&\mathrm{(B1)}\;\limsup\limits_{R\rightarrow+\infty}\left[X_{p,\alpha,\lambda}(R)+Y_{q,\beta,\mu}(R)+Y'_{2,\gamma}(R)\right]<+\infty,\text{ with }p\in\left(\frac{3}{2},3\right),\\
&\mathrm{(B2)}\;\lim\limits_{R\rightarrow+\infty}X_{p,\alpha}(R)=0,\;\limsup\limits_{R\rightarrow+\infty}\left[Y_{q,\beta,\mu}(R)+Y'_{2,\gamma}(R)\right]<+\infty,\text{ with }p=3.
\end{align*}
Then $u=v=0$ and $\theta$ is a constant.
\end{Thm}

As corollaries of Theorem \ref{main1}, we establish the following two interesting Liouville-type results under the integrability conditions imposed only on $(u,\nabla v)$ or $(\nabla u,\nabla v)$.
\begin{Cor}\label{Cor1.1}
Let $(u,\pi,v,\theta)$ be a smooth solution of \eqref{equ1.1}. Assume that $v$ satisfies the vanishing condition at infinity
\begin{align*}
\lim_{|x|\rightarrow+\infty}v(x)=0.
\end{align*}
If $u\in L^p(\mathbb{R}^3)$ with $\frac{3}{2}<p\leq3$ and $\nabla v\in L^2(\mathbb{R}^3)$, then $u=v=0$ and $\theta$ is a constant.
\end{Cor}

\begin{Cor}\label{Cor1.2}
Let $(u,\pi,v,\theta)$ be a smooth solution of \eqref{equ1.1}. Assume that $u$ and $v$ satisfy the vanishing condition at infinity
\begin{align*}
\lim_{|x|\rightarrow+\infty}u(x)=\lim_{|x|\rightarrow+\infty}v(x)=0.
\end{align*}
If $\nabla u\in L^p(\mathbb{R}^3)$ with $1\leq p\leq\frac{3}{2}$ and $\nabla v\in L^2(\mathbb{R}^3)$, then $u=v=0$ and $\theta$ is a constant.
\end{Cor}

We give some remarks on our results.
\begin{Rem}
Since we impose the conditions of finite Dirichlet energy and vanishing at infinity on $v$, Corollaries \ref{Cor1.1} and \ref{Cor1.2} are closely related to the Liouville-type problem proposed by Galdi.
\end{Rem}

\begin{Rem}
It is straightforward to verify that the triple \((\alpha,\beta,\gamma)\) defined in \eqref{abc} satisfies Assumptions \ref{a1.1} and \ref{a1.2} as a particular instance.
Consequently, the main conclusions established by Cho and Yang in \cite{CY26} can be regarded as a special corollary of Theorem \ref{main1}.
Compared with Cho-Yang's results, we achieve substantial improvements from three perspectives.
\begin{itemize}
\item[(i)] We broaden the admissible parameter domains of \(\alpha\) and \(\beta\). In \cite{CY26}, these two parameters are restricted to the fixed form \((\alpha,\beta)=\left(\frac{2}{p}-\frac{1}{3},\frac{2}{q}-\frac{1}{3}\right)\). The enhanced parameter flexibility enables the validity of Corollaries \ref{Cor1.1} and \ref{Cor1.2}.
\item[(ii)] We weaken the constraint on the growth exponent corresponding to \(\nabla v\): the original restriction \(\gamma=0\) is relaxed to \(\gamma\in[0,\gamma_{p,\alpha})\). In particular, \(\gamma\) is allowed to take arbitrarily large values when \(p=3\).
\item[(iii)] We sharpen the growth rates on $u$ and $v$ via logarithmic corrections.
\end{itemize}
\end{Rem}

Finally, we outline the arrangement of the rest of this paper. In Section \ref{sec2}, we recall key preliminary tools, including properties of the Bogovskii operator, the Poincar\'{e}-Sobolev inequality, the classical Poincar\'{e} inequality, a useful iteration lemma, and a crucial $L^2$ estimate for the oscillation of $\theta$ over the annular region. Section \ref{sec3} shows the proofs of Theorem \ref{main1}, Corollaries \ref{Cor1.1} and \ref{Cor1.2}, whereas Section \ref{sec4} is devoted to proving Theorem \ref{main2}.
Throughout this article, we use $C$ to denote a finite inessential constant which may be different from line to line.
\v0.1in
\section{Preliminaries}\label{sec2}

The first lemma is the existence and boundedness of the Bogovskii map, which is used to deal with the pressure term $\nabla \pi$.
\begin{Lem}\label{Lem2.1}$($See {\rm \cite[Lemma 1]{Tsai21}} or {\rm\cite[Theorem III.3.3]{Galdi}}$)$
Let $\Omega$ be a bounded Lipschitz domain in $\mathbb{R}^{3}$. Denote $L_{0}^p(\Omega):=\{g\in L^p(\Omega):\int_\Omega gdx=0\}$ with $1<p<+\infty$. There exists a linear operator
\begin{equation*}
\mathrm{Bog}:L_{0}^p(\Omega)\rightarrow W_{0}^{1,p}(\Omega),
\end{equation*}
such that for any  $g\in L_{0}^p(\Omega),w=\mathrm{Bog}g$ is a vector field satisfying
\begin{equation*}
w\in W_{0}^{1,p}(\Omega), \hspace{0.3cm}\mathrm{div}w=g,\hspace{0.3cm}\|\nabla w\|_{L^p(\Omega)}\leq C_{\mathrm{Bog}}(\Omega,p)\|g\|_{L^p(\Omega)},
\end{equation*}
where the constant $C_{\mathrm{Bog}}(\Omega,p)$  is independent of $g$. By using a rescaling argument, we see
$$C_{\mathrm{Bog}}(R\Omega,p)=C_{\mathrm{Bog}}(\Omega,p),\text{ where $R\Omega=\{Rx:x\in \Omega\}$ and $R>0$.}$$
\end{Lem}

The second lemma is the so-called Poincar\'{e}-Sobolev inequality. Let $(g)_{\Omega}$ represent the mean value of the function $g$ over $\Omega$.
\begin{Lem}\label{Lem2.2}$($See {\rm\cite[Theorem 3.15]{Giusti}}$)$
Let $\Omega\subset\mathbb{R}^n$ be a bounded Lipschitz domain. Assume $1\leq p<n$. Then there exists a positive constant $C(n,p,\Omega)$ such that
$$\|g-(g)_{\Omega}\|_{L^{\frac{np}{n-p}}(\Omega)}\leq C(n,p,\Omega)\|\nabla g\|_{L^{p}(\Omega)}\text{ for every $g\in W^{1,p}(\Omega)$.}$$
By using a rescaling argument, we see that the constant $C(n,p,R\Omega)$ does not depend on $R$, i.e.,  $C(n,p,R\Omega)=C(n,p,\Omega)$.
\end{Lem}

The third lemma is the so-called Poincar\'{e} inequality.
\begin{Lem}\label{Lem2.3}$($See {\rm\cite[p.292, Theorem 1]{Evans}}$)$
Let $\Omega\subset\mathbb{R}^n$ be a bounded $C^1$ domain. Assume $1\leq p\leq +\infty$. Then there exists a positive constant $C(n,p,\Omega)$ such that
$$\|g-(g)_{\Omega}\|_{L^p(\Omega)}\leq C(n,p,\Omega)\|\nabla g\|_{L^{p}(\Omega)}\text{ for every $g\in W^{1,p}(\Omega)$.}$$
By using a rescaling argument, we see that $C(n,p,R\Omega)=RC(n,p,\Omega)$.
\end{Lem}

Next, we show the following standard iteration lemma, which is a generalization of \cite[Lemma 3.1]{Giaquinta}, and can be found in \cite[Lemma 2.1]{CL24}.
\begin{Lem}\label{Lem2.4}
Let $f(t)$ be a non-negative bounded function on $\left[r_0, r_1\right] \subset \mathbb{R}^{+}$. If there are non-negative constants $a_i, b_i,\alpha_i$, $i=1,2,\cdots,m$, and a parameter $\kappa_0 \in[0,1)$ such that for any $r_0 \leq s<t \leq r_1$, it holds that
$$
f(s) \leq \kappa_0 f(t)+\sum_{i=1}^m\left(\frac{a_i}{(t-s)^{\alpha_i}}+b_i\right),
$$
then
$$
f(s) \leq C\sum_{i=1}^m\left(\frac{a_i}{(t-s)^{\alpha_i}}+b_i\right),
$$
where $C$ is a constant depending on $\alpha_1,\alpha_2,\cdots,\alpha_m$ and $\kappa_0$.
\end{Lem}

Finally, we introduce an important $L^2$ estimate for the oscillation of $\theta$ over the annular region, which plays a crucial role in eliminating the assumption on the temperature.
\begin{Lem}\label{Lem2.5}
Let $(u,\pi,v,\theta)$ be a smooth solution of \eqref{equ1.1}. For each $L>1$, $q\geq1$ and $\rho>0$, there exists a positive constant $C_L$  such that
\begin{align}\label{ine2.1}
\|\theta - (\theta)_{E}\|_{L^2(E)} \leq C_L \|u\|_{L^3(E)} \left( \|\nabla v\|_{L^2(E)} + \rho^{\frac{1}{2}-\frac{3}{q}} \|v\|_{L^q(E)} \right) + C_L\|\nabla v\|_{L^2(E)},
\end{align}
where $E= B_{L\rho}\backslash B_\rho$.
\end{Lem}
\begin{proof}
From \cite[Proposition 3]{CY26}, we see that for each $L>1$ and $\rho>0$, there exists a positive constant $C_L$  such that
\begin{align}\label{ine2.2}
\|\theta - (\theta)_{E}\|_{L^2(E)} \leq C_L \|u\|_{L^3(E)} \left( \|\nabla v\|_{L^2(E)} + \rho^{-1} \|v\|_{L^2(E)} \right) + C_L\|\nabla v\|_{L^2(E)}.
\end{align}
By the Minkowski inequality, the Poincar\'{e} inequality and the H\"{o}lder inequality, we derive
	\begin{align}\label{ine2.3}
		\|v\|_{L^2(E)}&\leq \|v-(v)_E\|_{L^2(E)}+\|(v)_E\|_{L^2(E)}\notag\\
		&\leq C\rho\|\nabla v\|_{L^2 (E)}+C\rho^\frac{3}{2}\left|(v)_E\right|\\
		&\leq C\rho\|\nabla v\|_{L^2 (E)}+C\rho^{\frac{3}{2}-\frac{3}{q}}\|v\|_{L^q(E)}.\notag
	\end{align}
Thus, \eqref{ine2.1} is an immediate consequence of \eqref{ine2.2} and \eqref{ine2.3}.
\end{proof}

\section{Proofs of Theorem \ref{main1} and Corollaries \ref{Cor1.1}, \ref{Cor1.2}  }\label{sec3}
Throughout this section, we adopt the following notations:
$$S_t=B_t\backslash \overline{B_\frac{\sqrt{2}t}{2}},\;\Theta=\theta-(\theta)_{S_t}.$$
For any $\rho>0$, we denote
\begin{equation}\label{ine3.1}
f(\rho)=\|u\|_{L^6 (B_\rho)}^2+\|v\|_{L^6 (B_\rho)}^2+\|\nabla u\|_{L^2 (B_\rho)}^2+\|\nabla v\|_{L^2 (B_\rho)}^2+\|\nabla \theta\|_{L^2 (B_\rho)}^2.
\end{equation}
To prove Theorem \ref{main1}, we first derive a crucial energy estimate for the energy function defined above.
\begin{Lem}\label{Lem3.1}
Let $(u,\pi,v,\theta)$ be a smooth solution of \eqref{equ1.1} and $\sqrt{2}R\leq s<t\leq 2R$.
Then we have
\begin{align}
f(s)\leq&\frac{1}{2}\int_{S_t}|\nabla u|^2dx+\frac{C}{(t-s)^2}\int_{S_t}(|u|^2+|v|^2) d x+\frac{C}{(t-s)^2}\int_{S_t}|\Theta|^2 d x\notag\\
&+\frac{C}{t-s} \int_{S_t}|u|^3  d x+\frac{C}{t-s}\|u\|_{L^p(S_t)}\|v\|_{L^{2p'}(S_t)}^2\label{ine3.2}\\
&+\frac{C}{t-s}\int_{S_t}|u||\Theta|^2dx+\frac{C}{t-s}\int_{S_t}|v||\Theta|dx,\notag
\end{align}
where $p>1$ and $p'$ is the conjugate exponent to $p$, i.e., $p'=\frac{p}{p-1}$.
\end{Lem}
\begin{proof}
Since $\sqrt{2}R\leq s<t\leq 2R$, we infer $s\geq\frac{\sqrt{2}t}{2}$. We introduce a cut-off function $\eta \in C_0^{\infty}\left(\mathbb{R}^{3}\right)$ satisfying
\begin{equation*}
\eta(x)= \begin{cases}1, & |x| <s, \\ 0, & |x| >\frac{s+t}{2},\end{cases}
\end{equation*}
with
$$\text{$0\leq\eta (x)\leq 1$, and $\|\nabla \eta\|_{L^{\infty}} \leq \frac{C}{t-s}$, $\|\nabla^2\eta\|_{L^{\infty}} \leq \frac{C}{(t-s)^2}$.}$$
By Lemma \ref{Lem2.1}, there exists $w\in W_{0}^{1,\sigma}(S_t)$ such that $w$ satisfies
$$
\mathrm{div }w=u\cdot\nabla\eta^{2} \text{ in }S_t,
$$
with the estimate
\begin{equation}\label{ine3.3}
\aligned
\|\nabla w\|_{L^\sigma\left(S_t\right)}\leq C\|u\cdot\nabla\eta^2\|_{L^\sigma\left(S_t\right)}\leq\frac{C}{t-s}\|u\|_{L^\sigma(S_t)},
\endaligned
\end{equation}
for any $1<\sigma<+\infty$. We extend $w$ by zero to $B_\frac{\sqrt{2}t}{2}$, then $w\in W_{0}^{1,\sigma}(B_t).$

Obviously, $(u,\pi,v,\Theta)$ also satisfies \eqref{equ1.1}. Multiply both sides of $\eqref{equ1.1}_{1}$, $\eqref{equ1.1}_{2}$ and $\eqref{equ1.1}_{3}$ by $u \eta^2-w$, $v \eta^2$ and $\Theta\eta^2$, respectively,
integrate over $B_t$ and apply integration by parts. This procedure yields
\begin{align}
&\int_{B_t}\left(|\nabla u|^{2}+|\nabla v|^{2}+|\nabla \theta|^{2}\right)\eta^2 d x\notag\\
=&\frac{1}{2}\int_{B_t}(|u|^2+|v|^2+|\Theta|^2) \Delta \eta^2 d x+\frac{1}{2} \int_{B_t}(|u|^2+|v|^2+|\Theta|^2) u \cdot \nabla \eta^2 d x\notag\\
&+\int_{B_t}\nabla u:\nabla w dx+ \int_{B_t}(v\otimes v):(u\otimes\nabla \eta^2) dx+ \int_{B_t}\Theta v\cdot\nabla \eta^2 dx\label{ine3.4}\\
& - \int_{B_t}(u\otimes u+v\otimes v): \nabla w  dx.\notag
\end{align}
Applying the Gagliardo-Nirenberg inequality, we have
\begin{align*}
\|u \eta\|_{L^6 (B_t)}^2+\|v \eta\|_{L^6 (B_t)}^2\leq&C\left(\|\nabla(u \eta)\|_{L^2 (B_t)}^2+\|\nabla(v \eta)\|_{L^2 (B_t)}^2\right)\\
\leq&C \left(\|\eta\nabla u\|_{L^2 (B_t)}^2+\|\eta\nabla v\|_{L^2 (B_t)}^2\right)\\
&+C\left(\|u\otimes\nabla \eta\|_{L^2 (B_t)}^2+\|v\otimes\nabla \eta\|_{L^2 (B_t)}^2\right).
\end{align*}
Combining the above inequality and \eqref{ine3.4}, we have
\begin{align}
&\|u\eta\|_{L^6 (B_t)}^2+\|v \eta\|_{L^6 (B_t)}^2+\int_{B_t}\left(|\nabla u|^{2}+|\nabla v|^{2}+|\nabla \theta|^{2}\right)\eta^2 d x\notag\\
\leq&\frac{C}{(t-s)^2}\int_{S_t}(|u|^2+|v|^2+|\Theta|^2) d x+ \frac{C}{t-s} \int_{S_t}|u|^3  d x\notag\\
&+C\int_{S_t}|\nabla u||\nabla w| dx+C\int_{S_t}|u|^{2} |\nabla w|dx+C\int_{S_t}|v|^{2}|\nabla w|dx \label{ine3.5}\\
&+\frac{C}{t-s}\int_{S_t}|v|^2|u|dx+\frac{C}{t-s}\int_{S_t}|\Theta|^2|u|dx+\frac{C}{t-s}\int_{S_t}|v||\Theta|dx\notag\\
=&:\sum_{i=1}^8I_i.\notag
\end{align}

By the H\"{o}lder inequality and \eqref{ine3.3}, we obtain
\begin{align}\label{ine3.6}
I_2+I_4\leq\frac{C}{t-s}\|u\|_{L^{3}(S_t)}^3+C\|u\|_{L^{3}(S_t)}^{2}\|\nabla w\|_{L^{3}(S_t)}\leq\frac{C}{t-s}\|u\|_{L^{3}(S_t)}^3.
\end{align}
By the Young inequality and \eqref{ine3.3}, we have
\begin{align}\label{ine3.7}
I_3&\leq\frac{1}{2}\int_{S_t}|\nabla u|^2dx+C\int_{S_t}|\nabla w|^2dx\leq\frac{1}{2}\int_{S_t}|\nabla u|^2dx+\frac{C}{(t-s)^2}\int_{S_t}|u|^2dx.
\end{align}
Using the H\"{o}lder inequality and \eqref{ine3.3} again, we obtain
\begin{align}\label{ine3.8}
I_5+I_6\leq&C\|v\|_{L^{2p'}(S_t)}^2\|\nabla w\|_{L^p(S_t)}+\frac{C}{t-s}\|v\|_{L^{2p'}(S_t)}^2\|u\|_{L^p(S_t)}\notag\\
\leq&\frac{C}{t-s}\|u\|_{L^p(S_t)}\|v\|_{L^{2p'}(S_t)}^2.
\end{align}
Collecting \eqref{ine3.5}, \eqref{ine3.6}, \eqref{ine3.7} and \eqref{ine3.8}, we conclude that \eqref{ine3.2} holds.
\end{proof}

For the sake of convenience, we denote six terms in \eqref{ine3.2} by
\begin{align*}
&J_1=\frac{C}{(t-s)^2}\int_{S_t}(|u|^2+|v|^2) d x,\quad J_2=\frac{C}{(t-s)^2}\int_{S_t}|\Theta|^2 d x,\notag\\
&J_3=\frac{C}{t-s} \int_{S_t}|u|^3  d x,\quad J_4=\frac{C}{t-s}\|u\|_{L^p(S_t)}\|v\|_{L^{2p'}(S_t)}^2,\\
&J_5=\frac{C}{t-s}\int_{S_t}|u||\Theta|^2dx,\quad J_6=\frac{C}{t-s}\int_{S_t}|v||\Theta|dx.\notag
\end{align*}
The estimates for $J_1$, $J_2$, $J_3$, $J_4$, $J_5$, $J_6$ are given in the next six lemmas.

\begin{Lem}\label{Lem3.2}
Let $\sqrt{2}R\leq s<t\leq 2R$, $p\geq1$ and $1\leq q\leq2$.
Suppose that $u,v$ are smooth vector-valued functions. Then we have
\begin{align}
J_1\leq&\frac{1}{16}\left(\|u\|_{L^{6}(B_t)}^{2}+\|v\|_{L^{6}(B_t)}^{2}\right)+\frac{C}{(t-s)^{\frac{6}{p}-1}}\|u\|_{L^{p}(A_{R})}^{2}\notag\\
&+\frac{CR^{3-\frac{6}{p}}}{(t-s)^2}\|u\|_{L^p(A_{R})}^2+\frac{C}{(t-s)^{\frac{6}{q}-1}}\|v\|_{L^{q}(A_{R})}^{2},\label{ine3.9}\\
J_1\leq&\frac{CR^2}{(t-s)^2}\left(\|u\|_{L^6(A_{R})}^2 +\|v\|_{L^6(A_{R})}^2\right).\label{ine3.10}
\end{align}
\end{Lem}

\begin{proof}
Denote
\begin{equation*}
J_{11}=\frac{C}{(t-s)^2}\int_{S_t}|u|^2 dx,\;J_{12}=\frac{C}{(t-s)^2}\int_{S_t}|v|^2dx,
\end{equation*}
then $J_1=J_{11}+J_{12}$.
When $1\leq p<2$, using the interpolation inequality and the Young inequality, we have
\begin{equation}\label{ine3.11}
\aligned
J_{11}&\leq\frac{C}{(t-s)^{2}}\|u\|_{L^p(S_t)}^{\frac{4p}{6-p}}\|u\|_{L^6(S_t)}^{\frac{12-6p}{6-p}}\\
&\leq\frac{1}{16}\|u\|_{L^{6}(S_t)}^{2}+\frac{C}{(t-s)^{\frac{6}{p}-1}}\|u\|_{L^{p}(S_t)}^{2}\\
&\leq\frac{1}{16}\|u\|_{L^{6}(B_t)}^{2}+\frac{C}{(t-s)^{\frac{6}{p}-1}}\|u\|_{L^{p}(A_R)}^{2}.
\endaligned
\end{equation}
When $p\geq2$, using the H\"{o}lder inequality, we obtain
\begin{equation}\label{ine3.12}
J_{11}\leq\frac{C}{(t-s)^2}t^{3-\frac{6}{p}}\|u\|_{L^p (S_t)}^2\leq\frac{C}{(t-s)^2}R^{3-\frac{6}{p}}\|u\|_{L^p (A_R)}^2.
\end{equation}
Combining \eqref{ine3.11} and \eqref{ine3.12}, we conclude that $J_{11}$ can be always controlled by
\begin{equation*}
J_{11}\leq\frac{1}{16}\|u\|_{L^{6}(B_t)}^{2}+\frac{C}{(t-s)^{\frac{6}{p}-1}}\|u\|_{L^{p}(A_R)}^{2}+\frac{C}{(t-s)^2}R^{3-\frac{6}{p}}\|u\|_{L^p (A_R)}^2.
\end{equation*}
By a similar argument, we conclude
\begin{equation*}
J_{12}\leq\frac{1}{16}\|v\|_{L^{6}(B_t)}^{2}+\frac{C}{(t-s)^{\frac{6}{q}-1}}\|v\|_{L^{q}(A_R)}^{2}.
\end{equation*}
Consequently, \eqref{ine3.9} holds.

Using the H\"{o}lder inequality, we obtain
\begin{align*}
J_1\leq\frac{Ct^2}{(t-s)^2}\left(\|u\|_{L^6(S_t)}^2 +\|v\|_{L^6(S_t)}^2\right)\leq\frac{CR^2}{(t-s)^2}\left(\|u\|_{L^6(A_R)}^2 +\|v\|_{L^6(A_R)}^2\right).
\end{align*}

\end{proof}

\begin{Lem}\label{Lem3.3}
Let $\sqrt{2}R\leq s<t\leq 2R$, $\frac{3}{2}<p\leq3$ and $1\leq q\leq2$. Let $(u,\pi,v,\theta)$ be a smooth solution of \eqref{equ1.1}. Then for any $\varepsilon\in\left(0,\frac{2p}{6-p}\right)$, there exist positive constants $C$ and $C_\varepsilon$ such that
\begin{align}
	J_2\leq&\frac{1}{8}\left(\|u\|_{L^6(B_t)}^2+\|\nabla v\|_{L^2(B_t)}^2\right)+\left[\frac{C R^{1-\frac{6}{q}}}{(t-s)^2}\|u\|_{L^p(A_R)}^\frac{2p}{6-p}\|v\|_{L^q(A_R)}^2\right]^\frac{6-p}{p}\notag\\
	&+\left[\frac{C_\varepsilon}{(t-s)^2}\|u\|_{L^p(A_R)}^\frac{2p}{6-p}\|\nabla v\|_{L^2(A_R)}^{\frac{12-4p}{6-p}+\varepsilon}\right]^\frac{2}{\varepsilon}+\left[\frac{C_\varepsilon}{(t-s)^2}\|\nabla v\|_{L^2(A_R)}^\varepsilon\right]^\frac{2}{\varepsilon},\label{ine3.13}\\
J_2\leq&\frac{CR^2}{(t-s)^2}\|\nabla \theta\|_{L^2(A_R)}^2.\label{ine3.14}
\end{align}

\end{Lem}
\begin{proof}
Taking $\rho=\frac{\sqrt{2}t}{2}$ and $L=\sqrt{2}$ in Lemma \ref{Lem2.5}, we obtain
\begin{align*}
 J_2&\leq\frac{C}{(t-s)^2}\|u\|_{L^3(S_t)}^2\|\nabla v\|_{L^2(S_t)}^2+\frac{Ct^{1-\frac{6}{q}}}{(t-s)^2}\|u\|_{L^3(S_t)}^2\|v\|_{L^q(S_t)}^2+\frac{C}{(t-s)^2}\|\nabla v\|_{L^2(S_t)}^2\\
 &=:J_{21}+J_{22}+J_{23}.
\end{align*}

When $p=3$, using the Young inequality, we get
\begin{align*}
J_{21}+J_{23}&=\|\nabla v\|_{L^2(S_t)}^{2-\varepsilon}\cdot\frac{C}{(t-s)^2}\left(\|u\|_{L^3(S_t)}^2+1\right)\|\nabla v\|_{L^2(S_t)}^\varepsilon\\
&\leq\frac{1}{16}\|\nabla v\|_{L^2(S_t)}^2+\left[\frac{C_\varepsilon}{(t-s)^2}\left(\|u\|_{L^3(S_t)}^2+ 1\right)\|\nabla v\|_{L^2(S_t)}^\varepsilon\right]^\frac{2}{\varepsilon}.
\end{align*}
Hence, we arrive at
\begin{align}
J_2\leq&\frac{1}{16}\|\nabla v\|_{L^2(B_t)}^2+\frac{CR^{1-\frac{6}{q}}}{(t-s)^2}\|u\|_{L^3(A_R)}^2\|v\|_{L^q(A_R)}^2\notag\\
&+\left[\frac{C_\varepsilon}{(t-s)^2}\|u\|_{L^3(A_R)}^2\|\nabla v\|_{L^2(A_R)}^\varepsilon\right]^\frac{2}{\varepsilon}+\left[\frac{C_\varepsilon}{(t-s)^2}\|\nabla v\|_{L^2(A_R)}^\varepsilon\right]^\frac{2}{\varepsilon}.\label{ine3.15}
\end{align}

When $\frac{3}{2} < p < 3$, by the interpolation inequality and the Young inequality, we get
\begin{align*}
J_{21}&\leq\frac{C}{(t-s)^2}\|u\|_{L^p(S_t)}^\frac{2p}{6-p}\|u\|_{L^6(S_t)}^\frac{12-4p}{6-p}\|\nabla v\|_{L^2(S_t)}^2\\
&=\|u\|_{L^6(S_t)}^\frac{12-4p}{6-p}\cdot\|\nabla v\|_{L^2(S_t)}^{\frac{2p}{6-p}-\varepsilon}\cdot\frac{C}{(t-s)^2}\|u\|_{L^p(S_t)}^\frac{2p}{6-p}\|\nabla v\|_{L^2(S_t)}^{\frac{12-4p}{6-p}+\varepsilon}\\
&\leq\frac{1}{16}\|u\|_{L^6(S_t)}^2+\frac{1}{16}\|\nabla v\|_{L^2(S_t)}^2+\left[\frac{C_\varepsilon}{(t-s)^2}\|u\|_{L^p(S_t)}^\frac{2p}{6-p}\|\nabla v\|_{L^2(S_t)}^{\frac{12-4p}{6-p}+\varepsilon}\right]^\frac{2}{\varepsilon},\\
J_{22}&\leq\|u\|_{L^6(S_t)}^\frac{12-4p}{6-p}\cdot\frac{CR^{1-\frac{6}{q}}}{(t-s)^2}\|u\|_{L^p(S_t)}^\frac{2p}{6-p}\|v\|_{L^q(S_t)}^2\\
&\leq\frac{1}{16}\|u\|_{L^6(S_t)}^2+\left[\frac{CR^{1-\frac{6}{q}}}{(t-s)^2}\|u\|_{L^p(S_t)}^\frac{2p}{6-p}\|v\|_{L^q(S_t)}^2\right]^\frac{6-p}{p},\\
J_{23}&=\|\nabla v\|_{L^2(S_t)}^{2-\varepsilon}\cdot\frac{C}{(t-s)^2}\|\nabla v\|_{L^2(S_t)}^\varepsilon\\
&\leq\frac{1}{16}\|\nabla v\|_{L^2(S_t)}^2+\left[\frac{C_\varepsilon}{(t-s)^2}\|\nabla v\|_{L^2(S_t)}^\varepsilon\right]^\frac{2}{\varepsilon}.
\end{align*}
Combining the above three inequalities, we deduce
\begin{align}
	J_2\leq&\frac{1}{8}\|u\|_{L^6(B_t)}^2+\frac{1}{8}\|\nabla v\|_{L^2(B_t)}^2+\left[\frac{C R^{1-\frac{6}{q}}}{(t-s)^2}\|u\|_{L^p(A_R)}^\frac{2p}{6-p}\|v\|_{L^q(A_R)}^2\right]^\frac{6-p}{p}\notag\\
&+\left[\frac{C_\varepsilon}{(t-s)^2}\|u\|_{L^p(A_R)}^\frac{2p}{6-p}\|\nabla v\|_{L^2(A_R)}^{\frac{12-4p}{6-p}+\varepsilon}\right]^\frac{2}{\varepsilon}+\left[\frac{C_\varepsilon}{(t-s)^2}\|\nabla v\|_{L^2(A_R)}^\varepsilon\right]^\frac{2}{\varepsilon}.\label{ine3.16}
\end{align}
Taking \eqref{ine3.15} and \eqref{ine3.16} into consideration, it is easy to verify that \eqref{ine3.13} is valid.

Using the Poincar\'{e} inequality, we get
\begin{align*}
	J_2\leq \frac{Ct^2}{(t-s)^2}\|\nabla \theta\|_{L^2(S_t)}^2\leq \frac{CR^2}{(t-s)^2}\|\nabla \theta\|_{L^2(A_R)}^2.
\end{align*}

\end{proof}

\begin{Lem}\label{Lem3.4}
Let $\sqrt{2}R\leq s<t\leq 2R$. Suppose that $u$ is a smooth vector-valued function. Then we have the following conclusions:
\begin{itemize}
\item[(i)]  For $p\in\left(\frac{3}{2},3\right]$, it holds that
\begin{equation}\label{ine3.17}
J_3\leq \frac{C}{t-s}\|u\|_{L^p(A_R)}^{\frac{3p}{6-p}}\|u\|_{L^6(A_R)}^{\frac{18-6p}{6-p}}.
\end{equation}
\item[(ii)] For $p\in\left(\frac{3}{2},3\right)$, it holds that
\begin{equation}\label{ine3.18}
J_3\leq\frac{1}{16}\|u\|_{L^6(B_t)}^2+\left(\frac{C}{(t-s)^{\frac{2}{p}-\frac{1}{3}}}\|u\|_{L^p(A_{R})}\right)^\frac{3p}{2p-3}.
\end{equation}
\end{itemize}
\end{Lem}

\begin{proof}
When $p\in\left(\frac{3}{2},3\right)$, using the interpolation inequality, we obtain
\begin{align}\label{ine3.19}
J_3\leq \frac{C}{t-s}\|u\|_{L^p(S_t)}^{\frac{3p}{6-p}}\|u\|_{L^6(S_t)}^{\frac{18-6p}{6-p}},
\end{align}
which implies \eqref{ine3.17}. Applying the Young inequality to \eqref{ine3.19}, we get
\begin{align*}
J_3&\leq \frac{1}{16}\|u\|_{L^6(S_t)}^2+\left(\frac{C}{t-s}\|u\|_{L^p(S_t)}^{\frac{3p}{6-p}}\right)^\frac{6-p}{2p-3}\\
&\leq\frac{1}{16}\|u\|_{L^6(B_t)}^2+\left(\frac{C}{(t-s)^{\frac{2}{p}-\frac{1}{3}}}\|u\|_{L^p(A_R)}\right)^\frac{3p}{2p-3}.
\end{align*}
\end{proof}

\begin{Lem}\label{Lem3.5}
Let $\sqrt{2}R\leq s<t\leq 2R$, $\frac{3}{2}<p\leq3$ and $1\leq q\leq2$. Suppose that $u,v$ are smooth vector-valued functions. Then we have
\begin{align}
J_4&\leq \frac{C}{t-s}\|u\|_{L^p(A_R)}\|v\|_{L^q(A_R)}^\frac{(4p-6)q}{(6-q)p}\|v\|_{L^6(A_R)}^{2-\frac{(4p-6)q}{(6-q)p}},\label{ine3.20}\\
J_4&\leq \frac{1}{16}\|v\|_{L^6(B_t)}^2+\left(\frac{C}{t-s}\|u\|_{L^p(A_R)}\|v\|_{L^q(A_R)}^\frac{(4p-6)q}{(6-q)p}\right)^\frac{(6-q)p}{(2p-3)q}.\label{ine3.21}
\end{align}
\end{Lem}
\begin{proof}
Using the interpolation inequality, we obtain
\begin{align}\label{ine3.22}
J_4&\leq \frac{C}{t-s}\|u\|_{L^p(S_t)}\|v\|_{L^q(S_t)}^\frac{(4p-6)q}{(6-q)p}\|v\|_{L^6(S_t)}^{2-\frac{(4p-6)q}{(6-q)p}},
\end{align}
which implies \eqref{ine3.20}. Applying the Young inequality to \eqref{ine3.22},
\begin{align*}
J_4&\leq \frac{1}{16}\|v\|_{L^6(S_t)}^2+\left(\frac{C}{t-s}\|u\|_{L^p(S_t)}\|v\|_{L^q(S_t)}^\frac{(4p-6)q}{(6-q)p}\right)^\frac{(6-q)p}{(2p-3)q},
\end{align*}
which yields \eqref{ine3.21}.
\end{proof}

\begin{Lem}\label{Lem3.6}
Let $\sqrt{2}R\leq s<t\leq 2R$, $\frac{3}{2}<p\leq3$ and $1\leq q\leq2$. Let $(u,\pi,v,\theta)$ be a smooth solution of \eqref{equ1.1}. Then for any $\varepsilon\in\left(0,\frac{2p-3}{6-p}\right)$, there exist positive constants $C_\varepsilon$  and $C$ such that
\begin{align}\label{ine3.23}
		J_5\leq&\frac{3}{16}\left(\|u\|_{L^6(S_t)}^2+\|\nabla v\|_{L^2(S_t)}^2+\|\nabla \theta\|_{L^2(S_t)}^2\right)+\left(\frac{C_\varepsilon}{t-s}\|u\|_{L^p(S_t)}^\frac{p+3}{6-p}\|\nabla v\|_{L^2(S_t)}^{\tau_1+\varepsilon}\right)^\frac{2}{\varepsilon}\notag\\
		&+\left(\frac{C R^{\tau_2}}{t-s}\|u\|_{L^p(S_t)}^\frac{p+3}{6-p}\|v\|_{L^q(S_t)}^{2-\frac{3}{p}}\right)^\frac{12-2p}{2p-3}+\left(\frac{C_\varepsilon}{t-s}\| u\|_{L^p(S_t)}\| \nabla v\|_{L^2(S_t)}^{\varepsilon}\right)^\frac{2}{\varepsilon},
\end{align}
where $\tau_1$ and $\tau_2$ are given by
$$\tau_1=\frac{6-2p}{6-p}\left(2-\frac{3}{p}\right),\quad \tau_2=\left(\frac{1}{2}-\frac{3}{q}\right)\left(2-\frac{3}{p}\right).$$
\end{Lem}
\begin{proof}
By the H\"older inequality, Lemma \ref{Lem2.5} and the Poincar\'{e}-Sobolev inequality, we have
\begin{align*}
J_5\leq&\frac{C}{t-s}\|u\|_{L^3(S_t)}\|\Theta\|_{L^2(S_t)}\|\Theta\|_{L^6(S_t)}\\
\leq&\frac{C}{t-s}\|u\|_{L^3(S_t)}^2\|\nabla v\|_{L^2(S_t)}\|\nabla\theta\|_{L^2(S_t)}+\frac{CR^{\frac{1}{2}-\frac{3}{q}}}{t-s}\|u\|_{L^3(S_t)}^2\|v\|_{L^q(S_t)}\|\nabla\theta\|_{L^2(S_t)}\\
&+\frac{C}{t-s}\|u\|_{L^3(S_t)}\|\nabla v\|_{L^2(S_t)}\|\nabla\theta\|_{L^2(S_t)} \\
:=&J_{51}+J_{52}+J_{53}.
\end{align*}

When $p=3$, using the Young inequality, we obtain
\begin{align*}
J_{51}+J_{53}&=\|\nabla v\|_{L^2(S_t)}^{1-\varepsilon}\cdot\|\nabla\theta\|_{L^2(S_t)}\cdot\frac{C}{t-s}\left(\|u\|_{L^3(S_t)}^2+\|u\|_{L^3(S_t)}\right)\|\nabla v\|_{L^2(S_t)}^\varepsilon\\
&\leq\frac{1}{16}\|\nabla v\|_{L^2(S_t)}^2+\frac{1}{16}\|\nabla\theta\|_{L^2(S_t)}^2+\left[\frac{C_\varepsilon}{t-s}\left(\|u\|_{L^3(S_t)}^2+\|u\|_{L^3(S_t)}\right)\|\nabla v\|_{L^2(S_t)}^\varepsilon\right]^\frac{2}{\varepsilon},\\
J_{52}&\leq\frac{1}{16}\|\nabla\theta\|_{L^2(S_t)}^2 +\frac{CR^{1-\frac{6}{q}}}{(t-s)^2}\|u\|_{L^3(S_t)}^4\|v\|_{L^q(S_t)}^2.
\end{align*}
Therefore, it follows that
\begin{align}
J_5\leq&\frac{1}{8}\|\nabla v\|_{L^2(S_t)}^2+\frac{1}{8}\|\nabla\theta\|_{L^2(S_t)}^2+\left[\frac{C_\varepsilon}{t-s}\left(\|u\|_{L^3(S_t)}^2+\|u\|_{L^3(S_t)}\right)\|\nabla v\|_{L^2(S_t)}^\varepsilon\right]^\frac{2}{\varepsilon}\notag\\
&+\frac{CR^{1-\frac{6}{q}}}{(t-s)^2}\|u\|_{L^3(S_t)}^4\|v\|_{L^q(S_t)}^2.\label{ine3.24}
\end{align}

When $\frac{3}{2} < p < 3$, employing the H\"{o}lder inequality, the interpolation inequality, Lemma \ref{Lem2.5} and the Poincar\'{e}-Sobolev inequality, we get
\begin{align*}
J_5\leq&\frac{C}{t-s}\|u\|_{L^p(S_t)}\|\Theta\|_{L^{2p'}(S_t)}^2\\
\leq&\frac{C}{t-s}\|u\|_{L^p(S_t)}\|\Theta\|_{L^2(S_t)}^{2-\frac{3}{p}}\|\Theta\|_{L^6(S_t)}^{\frac{3}{p}}\\
\leq&\frac{C}{t-s}\|u\|_{L^p(S_t)}\left[\|u\|_{L^3(S_t)}\left(\|\nabla v\|_{L^2(S_t)}+t^{\frac{1}{2}-\frac{3}{q}}\|v\|_{L^q(S_t)}\right)+\|\nabla v\|_{L^2(S_t)}\right]^{2-\frac{3}{p}}\|\nabla\theta\|_{L^2(S_t)}^{\frac{3}{p}}\\
\leq&\frac{C}{t-s}\|u\|_{L^p(S_t)}\left(\|u\|_{L^3(S_t)}\|\nabla v\|_{L^2(S_t)}\right)^{2-\frac{3}{p}}\|\nabla\theta\|_{L^2(S_t)}^{\frac{3}{p}}\\
&+\frac{C}{t-s}\|u\|_{L^p(S_t)}\left(t^{\frac{1}{2}-\frac{3}{q}}\|u\|_{L^3(S_t)}\|v\|_{L^q(S_t)}\right)^{2-\frac{3}{p}}\|\nabla\theta\|_{L^2(S_t)}^{\frac{3}{p}}\\
&+\frac{C}{t-s}\|u\|_{L^p(S_t)}\|\nabla v\|_{L^2(S_t)}^{2-\frac{3}{p}}\|\nabla\theta\|_{L^2(S_t)}^{\frac{3}{p}}\\
=&:J_{54}+J_{55}+J_{56}.
\end{align*}
By the interpolation inequality and the Young inequality, we have
\begin{align*}
J_{54}&\leq\frac{C}{t-s}\|u\|_{L^p(S_t)}\left(\|u\|_{L^p(S_t)}^{\frac{p}{6-p}}\|u\|_{L^6(S_t)}^{\frac{6-2p}{6-p}}\|\nabla v\|_{L^2(S_t)}\right)^{2-\frac{3}{p}}\|\nabla\theta\|_{L^2(S_t)}^{\frac{3}{p}}\\
&=\|u\|_{L^6(S_t)}^{\tau_1}\cdot\|\nabla v\|_{L^2(S_t)}^{\frac{2p-3}{6-p}-\varepsilon}\cdot\|\nabla\theta\|_{L^2(S_t)}^{\frac{3}{p}}\cdot\frac{C}{t-s}\|u\|_{L^p(S_t)}^\frac{3+p}{6-p}\|\nabla v\|_{L^2(S_t)}^{\tau_1+\varepsilon}\\
&\leq\frac{1}{16}\|u\|_{L^6(S_t)}^2+\frac{1}{16}\|\nabla v\|_{L^2(S_t)}^2+\frac{1}{16}\|\nabla\theta\|_{L^2(S_t)}^2+\left(\frac{C_\varepsilon}{t-s}\|u\|_{L^p(S_t)}^\frac{3+p}{6-p}\|\nabla v\|_{L^2(S_t)}^{\tau_1+\varepsilon}\right)^\frac{2}{\varepsilon},
\end{align*}
\begin{align*}
J_{55}&\leq\frac{C}{t-s}\|u\|_{L^p(S_t)}\left(t^{\frac{1}{2}-\frac{3}{q}}\|u\|_{L^p(S_t)}^{\frac{p}{6-p}}\|u\|_{L^6(S_t)}^{\frac{6-2p}{6-p}}\|v\|_{L^q(S_t)} \right)^{2-\frac{3}{p}}\|\nabla\theta\|_{L^2(S_t)}^{\frac{3}{p}}\\
&=\|u\|_{L^6(S_t)}^{\tau_1}\cdot\|\nabla\theta\|_{L^2(S_t)}^{\frac{3}{p}}\cdot\frac{CR^{\tau_2}}{t-s}\|u\|_{L^p(S_t)}^\frac{3+p}{6-p}\|v\|_{L^q(S_t)}^{2-\frac{3}{p}}\\
&\leq\frac{1}{16}\|u\|_{L^6(S_t)}^2+\frac{1}{16}\|\nabla\theta\|_{L^2(S_t)}^2+\left(\frac{C R^{\tau_2}}{t-s}\|u\|_{L^p(S_t)}^\frac{3+p}{6-p}\|v\|_{L^q(S_t)}^{2-\frac{3}{p}}\right)^\frac{12-2p}{2p-3},
\end{align*}
\begin{align*}
J_{56}&=\|\nabla v\|_{L^2(S_t)}^{2-\frac{3}{p}-\varepsilon}\cdot\|\nabla\theta\|_{L^2(S_t)}^{\frac{3}{p}}\cdot\frac{C}{t-s}\|u\|_{L^p(S_t)}\|\nabla v\|_{L^2(S_t)}^{\varepsilon}\\
&\leq\frac{1}{16}\|\nabla v\|_{L^2(S_t)}^2+\frac{1}{16}\|\nabla\theta\|_{L^2(S_t)}^2+\left(\frac{C_\varepsilon}{t-s}\|u\|_{L^p(S_t)}\|\nabla v\|_{L^2(S_t)}^{\varepsilon}\right)^\frac{2}{\varepsilon}.
\end{align*}
Therefore, we conclude
\begin{align}
	J_5\leq&\frac{3}{16}\left(\|u\|_{L^6(S_t)}^2+\|\nabla v\|_{L^2(S_t)}^2+\|\nabla \theta\|_{L^2(S_t)}^2\right)+\left(\frac{C_\varepsilon}{t-s}\|u\|_{L^p(S_t)}^\frac{p+3}{6-p}\|\nabla v\|_{L^2(S_t)}^{\tau_1+\varepsilon}\right)^\frac{2}{\varepsilon}\notag\\
	&+\left(\frac{CR^{\tau_2}}{t-s}\|u\|_{L^p(S_t)}^\frac{p+3}{6-p}\|v\|_{L^q(S_t)}^{2-\frac{3}{p}}\right)^\frac{12-2p}{2p-3}+\left(\frac{C_\varepsilon}{t-s}\| u\|_{L^p(S_t)}\| \nabla v\|_{L^2(S_t)}^{\varepsilon}\right)^\frac{2}{\varepsilon}.\label{ine3.25}
\end{align}
Combining \eqref{ine3.24} and \eqref{ine3.25}, we see that \eqref{ine3.23} holds.
\end{proof}

\begin{Lem}\label{Lem3.7}
	Let $\sqrt{2}R\leq s<t\leq 2R$. Let $(u,\pi,v,\theta)$ be a smooth solution of \eqref{equ1.1}. Then we have
\begin{itemize}
\item[(i)] For $(p,q)\in\left(\frac{3}{2},3\right]\times [1,2]$, it holds that
\begin{align}
	J_6 \leq&\frac{C}{t-s}\|u\|_{L^p(A_R)}^{\frac{p}{6-p}}\|v\|_{L^q(A_R)}^{\frac{2q}{6-q}} \|u\|_{L^6(A_R)}^{\frac{6-2p}{6-p}}  \|v\|_{L^6(A_R)}^{\frac{6-3q}{6-q}}\|\nabla v\|_{L^2(A_R)}\notag\\
	&+\frac{C}{t-s}\|v\|_{L^q(A_R)}^{\frac{2q}{6-q}} \|v\|_{L^6(A_R)}^{\frac{6-3q}{6-q}}\|\nabla v\|_{L^2(A_R)}\label{ine3.26a}\\
	&+\frac{C R^{\frac{1}{2}-\frac{3}{q}}}{t-s} \|u\|_{L^p(A_R)}^{\frac{p}{6-p}}\|v\|_{L^q(A_R)}^{\frac{6+q}{6-q}}\|u\|_{L^6(A_R)}^{\frac{6-2p}{6-p}} \|v\|_{L^6(A_R)}^{\frac{6-3q}{6-q}}.\notag
\end{align}
		\item[(ii)] For $(p,q)\in\left(\frac{3}{2},3\right]\times [1,2]$ and any $\varepsilon\in\left(0,2-\frac{6-2p}{6-p}-\frac{6-3q}{6-q}\right)$, there exist positive constants $C_\varepsilon$ and $C$ such that
		\begin{align}
				J_{6} \leq& \frac{3}{16}\left(\|u\|_{L^6(S_t)}^2 + \|v\|_{L^6(S_t)}^2 + \|\nabla v\|_{L^2(S_t)}^2\right)
				+\chi_1(p,q)\left(\frac{C}{t-s}\|u\|_{L^p(S_t)}^{\frac{p}{6-p}} \|v\|_{L^q(S_t)}^{\frac{2q}{6-q}}\right)^{\tau_3}\notag\\
&+\chi_2(p,q)\left( \frac{C_\varepsilon}{t-s}\|u\|_{L^p(S_t)}^{\frac{p}{6-p}} \|v\|_{L^q(S_t)}^{\frac{2q}{6-q}} \|\nabla v\|_{L^2(S_t)}^{\left(\frac{6-2p}{6-p}+\frac{6-3q}{6-q}-1+\varepsilon\right)\chi_2(p,q)} \right)^{\frac{2}{\varepsilon}}\label{ine3.27a}\\
				&+\left( \frac{C R^{\frac{1}{2}-\frac{3}{q}}}{t-s} \|u\|_{L^p(S_t)}^{\frac{p}{6-p}} \|v\|_{L^q(S_t)}^{\frac{6+q}{6-q}} \right)^{\tau_4}+ \left( \frac{C}{t-s}  \|v\|_{L^q(S_t)}^{\frac{2q}{6-q}} \right)^{\frac{6-q}{q}},\notag
		\end{align}
where $\chi_1(p,q)$ is defined by
\begin{equation*}
	\chi_1(p,q)=
	\begin{cases}
		1, & \text{ if  }\frac{6-2p}{6-p} + \frac{6-3q}{6-q} < 1, \\
		0, & \text{ if  }\frac{6-2p}{6-p} + \frac{6-3q}{6-q} \geq 1,
	\end{cases}
\end{equation*}
$\chi_2(p,q)=1-\chi_1(p,q)$, $\tau_3$ and $\tau_4$ are given by
$$
\tau_3=2\left[1-\left(\frac{6-2p}{6-p}+ \frac{6-3q}{6-q}\right)\chi_1(p,q)\right]^{-1},\;\tau_4=2\left(2-\frac{6-2p}{6-p} - \frac{6-3q}{6-q}\right)^{-1}.
$$
\end{itemize}
	
\end{Lem}
\begin{proof}
By the H\"older inequality and Lemma \ref{Lem2.5}, we have
	\begin{align}
		J_6&\leq\frac{C}{t-s}\|v\|_{L^2(S_t)}\|\Theta\|_{L^2(S_t)}\notag\\
		&\leq\frac{C}{t-s}\|v\|_{L^2(S_t)}\left(\|u\|_{L^3(S_t)}\|\nabla v\|_{L^2(S_t)} +t^{\frac{1}{2}-\frac{3}{q}}\|u\|_{L^3(S_t)}\|v\|_{L^q(S_t)}+\|\nabla v\|_{L^2(S_t)}\right).\label{ine3.29a}
	\end{align}
	
	When $p=3$ and $q=2$, by the Young inequality,  we obtain
	\begin{align}\label{ine3.30a}
		J_{6}
		\leq\frac{1}{16}\|\nabla v\|_{L^2(S_t)}^2+	\frac{C}{(t-s)^2}\|v\|_{L^2(S_t)}^2 \left(\|u\|_{L^3(S_t)}^2+1\right) + \frac{C R^{-1}}{t-s}\|u\|_{L^3(S_t)}\|v\|_{L^2(S_t)}^2.
	\end{align}

When $p\in\left(\frac{3}{2},3\right)$ and $q=2$, using the interpolation inequality,  we get
\begin{align}
	J_6 &\leq \frac{C}{t-s}\|v\|_{L^2(S_t)} \cdot
	\|u\|_{L^p(S_t)}^{\frac{p}{6-p}} \|u\|_{L^6(S_t)}^{\frac{6-2p}{6-p}} \|\nabla v\|_{L^2(S_t)}
	+\frac{C}{t-s}\|v\|_{L^2(S_t)} \cdot \|\nabla v\|_{L^2(S_t)}\notag\\
	&+\frac{C}{t-s}\|v\|_{L^2(S_t)}\cdot R^{-1}\|u\|_{L^p(S_t)}^{\frac{p}{6-p}} \|u\|_{L^6(S_t)}^{\frac{6-2p}{6-p}} \|v\|_{L^2(S_t)}\label{ine3.31a}\\
	&=: J_{61} + J_{62} + J_{63}.\notag
\end{align}
Using the Young inequality, we have
\begin{align*}
J_{61}\leq& \frac{1}{16}\|u\|_{L^6(S_t)}^2 + \frac{1}{16}\|\nabla v\|_{L^2(S_t)}^2 + \left( \frac{C}{t-s} \|u\|_{L^p(S_t)}^{\frac{p}{6-p}} \|v\|_{L^2(S_t)} \right)^{\frac{12-2p}{p}},\\
J_{62}\leq& \frac{1}{16}\|\nabla v\|_{L^2(S_t)}^2 + \frac{C}{(t-s)^2}\|v\|_{L^2(S_t)}^2,\\
J_{63}\leq& \frac{1}{16}\|u\|_{L^6(S_t)}^2 + \left( \frac{C R^{-1}}{t-s} \|u\|_{L^p(S_t)}^{\frac{p}{6-p}} \|v\|_{L^2(S_t)}^2 \right)^{\frac{6-p}{3}}.
\end{align*}
Thus, it follows that
\begin{align}
J_6 \leq& \frac{1}{8}\left(\|u\|_{L^6(S_t)}^2 + \|\nabla v\|_{L^2(S_t)}^2\right)
		+ \left( \frac{C}{t-s} \|u\|_{L^p(S_t)}^{\frac{p}{6-p}} \|v\|_{L^2(S_t)} \right)^{\frac{12-2p}{p}}\notag\\
&+ \frac{C}{(t-s)^2}\|v\|_{L^2(S_t)}^2+ \left( \frac{C R^{-1}}{t-s} \|u\|_{L^p(S_t)}^{\frac{p}{6-p}} \|v\|_{L^2(S_t)}^2 \right)^{\frac{6-p}{3}}.\label{ine3.32a}
\end{align}

	When $p=3$ and $q\in[1,2)$, applying the interpolation inequality to \eqref{ine3.29a}, we obtain
	\begin{align}
		J_6 \leq& \frac{C}{t-s}\|v\|_{L^q(S_t)}^{\frac{2q}{6-q}} \|v\|_{L^6(S_t)}^{\frac{6-3q}{6-q}} \left[\left(\|u\|_{L^3(S_t)}+1\right)\|\nabla v\|_{L^2(S_t)} + R^{\frac{1}{2}-\frac{3}{q}}\|u\|_{L^3(S_t)}\|v\|_{L^q(S_t)}\right].\label{ine3.33a}
	\end{align}
Then it follows from the Young inequality that
	\begin{align}
		J_6 \leq& \frac{1}{16}\left( \|v\|_{L^6(S_t)}^2 + \|\nabla v\|_{L^2(S_t)}^2 \right)
		+ \left[ \frac{C}{t-s}\left(\|u\|_{L^3(S_t)}+1\right)\|v\|_{L^q(S_t)}^{\frac{2q}{6-q}} \right]^{\frac{6-q}{q}}\notag\\
		&+\frac{1}{16}\|v\|_{L^6(S_t)}^2 + \left( \frac{C R^{\frac{1}{2}-\frac{3}{q}}}{t-s}\|u\|_{L^3(S_t)}\|v\|_{L^q(S_t)}^{\frac{6+q}{6-q}} \right)^{\frac{12-2q}{6+q}}\notag\\
       \leq& \frac{1}{8}\left( \|v\|_{L^6(S_t)}^2 + \|\nabla v\|_{L^2(S_t)}^2 \right)
		+ \left[ \frac{C}{t-s}\left(\|u\|_{L^3(S_t)}+1\right)\|v\|_{L^q(S_t)}^{\frac{2q}{6-q}} \right]^{\frac{6-q}{q}}\label{ine3.34a}\\
		&+ \left( \frac{C R^{\frac{1}{2}-\frac{3}{q}}}{t-s}\|u\|_{L^3(S_t)}\|v\|_{L^q(S_t)}^{\frac{6+q}{6-q}} \right)^{\frac{12-2q}{6+q}}. \notag
	\end{align}

	When  $p\in\left(\frac{3}{2},3\right)$ and $q\in[1,2)$, by the interpolation inequality,  we get
\begin{align}
	J_6 \leq& \frac{C}{t-s}\|v\|_{L^q(S_t)}^{\frac{2q}{6-q}} \|v\|_{L^6(S_t)}^{\frac{6-3q}{6-q}}
	\|u\|_{L^p(S_t)}^{\frac{p}{6-p}} \|u\|_{L^6(S_t)}^{\frac{6-2p}{6-p}} \|\nabla v\|_{L^2(S_t)}\notag\\
	&+\frac{CR^{\frac{1}{2}-\frac{3}{q}}}{t-s}\|v\|_{L^q(S_t)}^{\frac{6+q}{6-q}} \|v\|_{L^6(S_t)}^{\frac{6-3q}{6-q}}
	  \|u\|_{L^p(S_t)}^{\frac{p}{6-p}} \|u\|_{L^6(S_t)}^{\frac{6-2p}{6-p}}\notag\\
&+\frac{C}{t-s}\|v\|_{L^q(S_t)}^{\frac{2q}{6-q}} \|v\|_{L^6(S_t)}^{\frac{6-3q}{6-q}}\|\nabla v\|_{L^2(S_t)}\label{ine3.35a}\\
	&\mathrel{=:} J_{64} + J_{65} + J_{66}.\notag
\end{align}
By the Young inequality, we derive the estimates for $J_{65}$ and $J_{66}$:
\begin{align}
J_{65} \leq& \frac{1}{16}\left(\|u\|_{L^6(S_t)}^2 + \|v\|_{L^6(S_t)}^2\right)
+ \left( \frac{C R^{\frac{1}{2}-\frac{3}{q}}}{t-s} \|u\|_{L^p(S_t)}^{\frac{p}{6-p}} \|v\|_{L^q(S_t)}^{\frac{6+q}{6-q}} \right)^{\tau_4},\label{ine3.36a}\\
J_{66} \leq&\frac{1}{16}\left(\|v\|_{L^6(S_t)}^2 + \|\nabla v\|_{L^2(S_t)}^2\right)
+ \left( \frac{C}{t-s}  \|v\|_{L^q(S_t)}^{\frac{2q}{6-q}} \right)^{\frac{6-q}{q}}.\label{ine3.37a}
\end{align}
For the estimate of $J_{64}$, we split our analysis into two cases according to the relationship between parameters $p$ and $q$.
If $\frac{6-2p}{6-p} + \frac{6-3q}{6-q} < 1$, employing the Young inequality, we have
\begin{align}
J_{64} \leq \frac{1}{16}\left(\|u\|_{L^6(S_t)}^2 + \|v\|_{L^6(S_t)}^2 + \|\nabla v\|_{L^2(S_t)}^2\right)
+ \left(\frac{C}{t-s}\|u\|_{L^p(S_t)}^{\frac{p}{6-p}} \|v\|_{L^q(S_t)}^{\frac{2q}{6-q}}\right)^{\tau_3}.\label{ine3.38a}
\end{align}
If $\frac{6-2p}{6-p} + \frac{6-3q}{6-q} \geq 1$, using the Young inequality, we obtain
\begin{align}
	J_{64} =&\|u\|_{L^6(S_t)}^{\frac{6-2p}{6-p}} \|v\|_{L^6(S_t)}^{\frac{6-3q}{6-q}} \|\nabla v\|_{L^2(S_t)}^{2-\frac{6-2p}{6-p}-\frac{6-3q}{6-q}-\varepsilon} \cdot \frac{C}{t-s}\|u\|_{L^p(S_t)}^{\frac{p}{6-p}} \|v\|_{L^q(S_t)}^{\frac{2q}{6-q}} \|\nabla v\|_{L^2(S_t)}^{\frac{6-2p}{6-p}+\frac{6-3q}{6-q}-1+\varepsilon} \notag\\
	\leq&\frac{1}{16}\left(\|u\|_{L^6(S_t)}^2 + \|v\|_{L^6(S_t)}^2 + \|\nabla v\|_{L^2(S_t)}^2\right)\label{ine3.40a}\\
	&+ \left( \frac{C_\varepsilon}{t-s}\|u\|_{L^p(S_t)}^{\frac{p}{6-p}} \|v\|_{L^q(S_t)}^{\frac{2q}{6-q}} \|\nabla v\|_{L^2(S_t)}^{\frac{6-2p}{6-p}+\frac{6-3q}{6-q}-1+\varepsilon} \right)^{\frac{2}{\varepsilon}}.\notag
\end{align}
Combining \eqref{ine3.35a}, \eqref{ine3.36a}, \eqref{ine3.37a}, \eqref{ine3.38a} and \eqref{ine3.40a}, we conclude
\begin{align}
				J_{6} \leq& \frac{3}{16}\left(\|u\|_{L^6(S_t)}^2 + \|v\|_{L^6(S_t)}^2 + \|\nabla v\|_{L^2(S_t)}^2\right)
				+\chi_1(p,q)\left(\frac{C}{t-s}\|u\|_{L^p(S_t)}^{\frac{p}{6-p}} \|v\|_{L^q(S_t)}^{\frac{2q}{6-q}}\right)^{\tau_3}\notag\\
&+\chi_2(p,q)\left( \frac{C_\varepsilon}{t-s}\|u\|_{L^p(S_t)}^{\frac{p}{6-p}} \|v\|_{L^q(S_t)}^{\frac{2q}{6-q}} \|\nabla v\|_{L^2(S_t)}^{\left(\frac{6-2p}{6-p}+\frac{6-3q}{6-q}-1+\varepsilon\right)\chi_2(p,q)} \right)^{\frac{2}{\varepsilon}}\label{ine3.40}\\
				&+\left( \frac{C R^{\frac{1}{2}-\frac{3}{q}}}{t-s} \|u\|_{L^p(S_t)}^{\frac{p}{6-p}} \|v\|_{L^q(S_t)}^{\frac{6+q}{6-q}} \right)^{\tau_4}+ \left( \frac{C}{t-s}  \|v\|_{L^q(S_t)}^{\frac{2q}{6-q}} \right)^{\frac{6-q}{q}}.\notag
\end{align}

Combining \eqref{ine3.29a}, \eqref{ine3.31a}, \eqref{ine3.33a} and \eqref{ine3.35a}, we find that \eqref{ine3.26a} holds.
In view of \eqref{ine3.30a}, \eqref{ine3.32a}, \eqref{ine3.34a} and \eqref{ine3.40}, we obtain \eqref{ine3.27a}.
\end{proof}

With the above preparations, we proceed to prove Theorem \ref{main1}.
\begin{proof}[{\bf Proof of Theorem \ref{main1}}]
{\bf Assume that \eqref{c1.3}  holds.} Then there exists a sequence $R_j\nearrow+\infty$ such that
\begin{align}\label{ine3.41}
\lim\limits_{j\rightarrow+\infty}X_{p,\alpha}(R_j)<+\infty,\;\lim\limits_{j\rightarrow+\infty}Y_{q,\beta}(R_j)<+\infty,\;\lim\limits_{j\rightarrow+\infty}Y'_{2,\gamma}(R_j)<+\infty.
\end{align}
For convenience, we introduce the following eight exponents:
\begin{align*}
&\delta_1=\frac{2p}{6-p}\alpha+2\beta-\frac{6}{q}-1,\quad\delta_2(\varepsilon)=\frac{2p}{6-p}\alpha+\left(\frac{12-4p}{6-p}+\varepsilon\right)\gamma-2,\\
&\delta_3=\alpha+\frac{(4p-6)q}{(6-q)p}\beta-1,\quad\delta_4(\varepsilon)=\frac{p+3}{6-p}\alpha+(\tau_1+\varepsilon)\gamma-1,\\
&\delta_5=\frac{p+3}{6-p}\alpha+\left(2-\frac{3}{p}\right)\beta+\tau_2-1,\quad\delta_6=\frac{p}{6-p}\alpha+\frac{2q}{6-q}\beta-1,\\
&\delta_7(\varepsilon)=\frac{p}{6-p}\alpha+\frac{2q}{6-q}\beta+\left(\frac{6-2p}{6-p}+\frac{6-3q}{6-q}-1+\varepsilon\right)\chi_2(p,q)\gamma-1,\\
&\delta_8=\frac{p}{6-p}\alpha+\frac{6+q}{6-q}\beta-\frac{3}{q}-\frac{1}{2},
\end{align*}
where $\tau_1$ and $\tau_2$ are given in Lemma \ref{Lem3.6}, and $\chi_2$ is defined in Lemma \ref{Lem3.7}. By straightforward calculations, we can verify nonpositivity of
$\delta_1$, $\delta_2(0)$, $\delta_3$, $\delta_4(0)$, $\delta_5$, $\delta_6$, $\delta_7(0)$, $\delta_8$. Indeed, we have
\begin{align}
&\delta_1\leq\frac{2p}{6-p}\left(\frac{2}{p}-\frac{1}{3}\right)+2\left(\frac{3}{q}-\frac{1}{2}\right)-\frac{6}{q}-1=-\frac{4}{3},\notag\\
&\delta_2(0)\begin{cases}
<\frac{2p}{6-p}\alpha+\frac{12-4p}{6-p}\cdot\frac{6-p-p\alpha}{6-2p}-2=0, & \text{ if }\;\frac{3}{2}<p<3, \\
\leq\frac{2p}{6-p}\left(\frac{2}{p}-\frac{1}{3}\right)-2=-\frac{4}{3}, & \text{ if }\;p=3,
\end{cases}\notag\\
&\delta_3=\alpha+\frac{2p-3}{p}\cdot\frac{2q}{6-q}\beta-1\leq\alpha+\frac{2p-3}{p}\cdot\left(1-\frac{p}{6-p}\alpha\right)-1\notag\\
&\;\;\;=\frac{2p-3}{p}+\frac{9-3p}{6-p}\alpha-1\leq\frac{2p-3}{p}+\frac{9-3p}{6-p}\left(\frac{2}{p}-\frac{1}{3}\right)-1=0,\notag\\
&\delta_4(0)\begin{cases}
<\frac{p+3}{6-p}\alpha+\frac{6-2p}{6-p}\left(2-\frac{3}{p}\right)\cdot\frac{p\bigl[6-p-(p+3)\alpha\bigr]}
{(6-2p)(2p-3)}-1=0, & \text{ if }\;\frac{3}{2}<p<3, \\
\leq\frac{p+3}{6-p}\left(\frac{2}{p}-\frac{1}{3}\right)-1=-\frac{2p-3}{3p}=-\frac{1}{3}, & \text{ if }\;p=3,
\end{cases}\label{ine3.42}\\
&\delta_5=\frac{p+3}{6-p}\alpha+\left(2-\frac{3}{p}\right)\cdot\left[\beta-\left(\frac{3}{q}-\frac{1}{2}\right)\right]-1\leq\frac{p+3}{6-p}\left(\frac{2}{p}-\frac{1}{3}\right)-1=-\frac{2p-3}{3p},\notag\\
&\delta_6\leq0,\notag\\
&\delta_7(0)\begin{cases}
		=\delta_6, & \text{ if  }\frac{6-2p}{6-p} + \frac{6-3q}{6-q} < 1, \\
		=\delta_6, & \text{ if  }\frac{6-2p}{6-p} + \frac{6-3q}{6-q} \geq 1 \text{ and  }\gamma=0,\\
           <0, & \text{ if  }\frac{6-2p}{6-p} + \frac{6-3q}{6-q} \geq 1 \text{ and  }\gamma>0,
	\end{cases}\notag\\
&\delta_8\leq1-\frac{2q}{6-q}\beta+\frac{6+q}{6-q}\beta-\frac{3}{q}-\frac{1}{2}=\beta-\left(\frac{3}{q}-\frac{1}{2}\right)\leq0.\notag
\end{align}
Hence, we can choose a sufficiently small positive number $\varepsilon$ such that
\begin{align}\label{ine3.43}
\delta_2(\varepsilon)<0,\quad \delta_4(\varepsilon)<0,\quad \varepsilon\gamma-2<0,\quad \alpha+\varepsilon\gamma-1<0,
\end{align}
and
\begin{align}\label{ine3.43a}
\delta_7(\varepsilon)\begin{cases}
		=\delta_6\leq0, & \text{ if  }\frac{6-2p}{6-p} + \frac{6-3q}{6-q} < 1, \\
		=\delta_6\leq0, & \text{ if  }\frac{6-2p}{6-p} + \frac{6-3q}{6-q} \geq 1 \text{ and  }\gamma=0,\\
         =\delta_7(0)+\varepsilon\gamma <0, & \text{ if  }\frac{6-2p}{6-p} + \frac{6-3q}{6-q} \geq 1 \text{ and  }\gamma>0.
	\end{cases}
\end{align}
Combining  \eqref{ine3.2}, \eqref{ine3.9}, \eqref{ine3.13}, \eqref{ine3.18}, \eqref{ine3.21}, \eqref{ine3.23} and \eqref{ine3.27a}, we derive that
\begin{align*}
f(s)\leq&\frac{5}{8}f(t)+\frac{C}{(t-s)^{\frac{6}{p}-1}}\|u\|_{L^p(A_{R})}^{2}+\frac{C}{(t-s)^2}R^{3-\frac{6}{p}}\|u\|_{L^{p}(A_{R})}^{2}+\frac{C}{(t-s)^{\frac{6}{q}-1}}\|v\|_{L^{q}(A_{R})}^{2}\\
&+\left(\frac{C R^{1-\frac{6}{q}}}{(t-s)^2}\|u\|_{L^p(A_R)}^\frac{2p}{6-p}\|v\|_{L^q(A_R)}^2\right)^\frac{6-p}{p}+\left(\frac{C_\varepsilon}{(t-s)^2}\|u\|_{L^p(A_R)}^\frac{2p}{6-p}\|\nabla v\|_{L^2(A_R)}^{\frac{12-4p}{6-p}+\varepsilon}\right)^\frac{2}{\varepsilon}\\
&+\left(\frac{C_\varepsilon}{(t-s)^2}\|\nabla v\|_{L^2(A_R)}^\varepsilon\right)^\frac{2}{\varepsilon}+\left(\frac{C}{(t-s)^{\frac{2}{p}-\frac{1}{3}}}\|u\|_{L^p(A_{R})}\right)^\frac{3p}{2p-3}\\
&+\left(\frac{C}{t-s}\|u\|_{L^p(A_R)}\|v\|_{L^q(A_R)}^\frac{(4p-6)q}{(6-q)p}\right)^\frac{(6-q)p}{(2p-3)q}+\left(\frac{C_\varepsilon}{t-s}\|u\|_{L^p(A_R)}^\frac{p+3}{6-p}\|\nabla v\|_{L^2(A_R)}^{\tau_1+\varepsilon}\right)^\frac{2}{\varepsilon}\\
&+\left(\frac{C R^{\tau_2}}{t-s}\|u\|_{L^p(A_R)}^\frac{p+3}{6-p}\|v\|_{L^q(A_R)}^{2-\frac{3}{p}}\right)^\frac{12-2p}{2p-3}+\left(\frac{C_\varepsilon}{t-s}\| u\|_{L^p(A_R)}\| \nabla v\|_{L^2(A_R)}^{\varepsilon}\right)^\frac{2}{\varepsilon}\\
&+\chi_1(p,q)\left(\frac{C}{t-s}\|u\|_{L^p(A_R)}^{\frac{p}{6-p}} \|v\|_{L^q(A_R)}^{\frac{2q}{6-q}}\right)^{\tau_3}\\
&+\chi_2(p,q)\left( \frac{C_\varepsilon}{t-s}\|u\|_{L^p(A_R)}^{\frac{p}{6-p}} \|v\|_{L^q(A_R)}^{\frac{2q}{6-q}} \|\nabla v\|_{L^2(A_R)}^{\left(\frac{6-2p}{6-p}+\frac{6-3q}{6-q}-1+\varepsilon\right)\chi_2(p,q)} \right)^{\frac{2}{\varepsilon}}\\
&+\left( \frac{C R^{\frac{1}{2}-\frac{3}{q}}}{t-s} \|u\|_{L^p(A_R)}^{\frac{p}{6-p}} \|v\|_{L^q(A_R)}^{\frac{6+q}{6-q}} \right)^{\tau_4}+ \left(\frac{C}{t-s} \|v\|_{L^q(A_R)}^{\frac{2q}{6-q}} \right)^{\frac{6-q}{q}}.
\end{align*}
Applying Lemma \ref{Lem2.4} to the above function inequality, and taking $s=\sqrt{2}R$ and $t=2R$, we conclude that
\begin{align}\label{ine3.44}
f\left(\sqrt{2}R\right)&\leq CR^{1-\frac{6}{p}+2\alpha}X^2_{p,\alpha}(R)+CR^{1-\frac{6}{q}+2\beta}Y^2_{q,\beta}(R)+\left(CR^{\delta_1}X^\frac{2p}{6-p}_{p,\alpha}(R)Y^2_{q,\beta}(R)\right)^\frac{6-p}{p}\notag\\
&+\left(C_\varepsilon R^{\delta_2(\varepsilon)}X^\frac{2p}{6-p}_{p,\alpha}(R)[Y'_{2,\gamma}(R)]^{\frac{12-4p}{6-p}+\varepsilon}\right)^\frac{2}{\varepsilon}+\left(C_\varepsilon R^{\varepsilon\gamma-2}[Y'_{2,\gamma}(R)]^{\varepsilon}\right)^\frac{2}{\varepsilon}\notag\\
&+\left(CR^{\alpha-\left(\frac{2}{p}-\frac{1}{3}\right)}X_{p,\alpha}(R)\right)^\frac{3p}{2p-3}+\left(CR^{\delta_3}X_{p,\alpha}(R)Y^\frac{(4p-6)q}{(6-q)p}_{q,\beta}(R)\right)^\frac{(6-q)p}{(2p-3)q}\notag\\
&+\left(C_\varepsilon R^{\delta_4(\varepsilon)}X^\frac{p+3}{6-p}_{p,\alpha}(R)[Y'_{2,\gamma}(R)]^{\tau_1+\varepsilon}\right)^\frac{2}{\varepsilon}+\left(CR^{\delta_5}X^\frac{p+3}{6-p}_{p,\alpha}(R)Y^{2-\frac{3}{p}}_{q,\beta}(R)\right)^\frac{12-2p}{2p-3}\\
&+\left[C_\varepsilon R^{\alpha+\varepsilon\gamma-1}X_{p,\alpha}(R)[Y'_{2,\gamma}(R)]^{\varepsilon}\right]^\frac{2}{\varepsilon}+\chi_1(p,q)\left(CR^{\delta_6}X^\frac{p}{6-p}_{p,\alpha}(R)Y^\frac{2q}{6-q}_{q,\beta}(R)\right)^{\tau_3}\notag\\
&+\chi_2(p,q)\left(C_\varepsilon R^{\delta_7(\varepsilon)}X^\frac{p}{6-p}_{p,\alpha}(R)Y^\frac{2q}{6-q}_{q,\beta}(R)[Y'_{2,\gamma}(R)]^{\left(\frac{6-2p}{6-p}+\frac{6-3q}{6-q}-1+\varepsilon\right)\chi_2(p,q)}\right)^\frac{2}{\varepsilon}\notag\\
&+\left(CR^{\delta_8}X^\frac{p}{6-p}_{p,\alpha}(R)Y^\frac{6+q}{6-q}_{q,\beta}(R)\right)^{\tau_4}+\left(CR^{\frac{2q}{6-q}\beta-1}Y^\frac{2q}{6-q}_{q,\beta}(R)\right)^\frac{6-q}{q}.\notag
\end{align}
Letting $R=R_j\rightarrow+\infty$, and using \eqref{ine3.41}, \eqref{ine3.42}, \eqref{ine3.43}, \eqref{ine3.43a} and the following facts
\begin{align*}
1-\frac{6}{p}+2\alpha<0,\;1-\frac{6}{q}+2\beta\leq0,\;\alpha-\left(\frac{2}{p}-\frac{1}{3}\right)\leq0,\;\frac{2q}{6-q}\beta-1\leq0,
\end{align*}
we obtain $u,v\in L^6(\mathbb{R}^3)$ and $\nabla u,\nabla v,\nabla \theta\in L^2(\mathbb{R}^3)$. Furthermore, it holds that
\begin{equation}\label{ine3.45}
\aligned
&\lim_{R\rightarrow+\infty}\left(\|u\|_{L^6(A_{R})}+\|v\|_{L^6(A_{R})}+\|\nabla u\|_{L^2(A_{R})}+\|\nabla v\|_{L^2(A_{R})}+\|\nabla \theta\|_{L^2(A_{R})}\right)=0.
\endaligned
\end{equation}
Combining  \eqref{ine3.2}, \eqref{ine3.10}, \eqref{ine3.14}, \eqref{ine3.17}, \eqref{ine3.20}, \eqref{ine3.23} and \eqref{ine3.26a}, and taking $s=\sqrt{2}R$ and $t=2R$, we have
\begin{align*}
f\left(\sqrt{2}R\right)\leq&C\left(\|\nabla u\|_{L^{2}(A_{R})}^2+\|\nabla v\|_{L^{2}(A_{R})}^2+\|\nabla\theta\|_{L^2(A_{R})}^2+\|u\|_{L^6(A_{R})}^2 +\|v\|_{L^6(A_{R})}^2\right)\\
&+CR^{\frac{3p}{6-p}\alpha-1}X^{\frac{3p}{6-p}}_{p,\alpha}(R)\|u\|_{L^6(A_{R})}^{\frac{18-6p}{6-p}}+CR^{\delta_3}X_{p,\alpha}(R)Y^\frac{(4p-6)q}{(6-q)p}_{q,\beta}(R)\|v\|_{L^6(A_{R})}^{2-\frac{(4p-6)q}{(6-q)p}}\\
&+\left(C_\varepsilon R^{\delta_4(\varepsilon)}X^\frac{p+3}{6-p}_{p,\alpha}(R)[Y'_{2,\gamma}(R)]^{\tau_1+\varepsilon}\right)^\frac{2}{\varepsilon}+\left(CR^{\delta_5}X^\frac{p+3}{6-p}_{p,\alpha}(R)Y^{2-\frac{3}{p}}_{q,\beta}(R)\right)^\frac{12-2p}{2p-3}\\
&+\left(C_\varepsilon R^{\alpha+\varepsilon\gamma-1}X_{p,\alpha}(R)[Y'_{2,\gamma}(R)]^{\varepsilon}\right)^\frac{2}{\varepsilon}\\
&+CR^{\delta_6}X^\frac{p}{6-p}_{p,\alpha}(R)Y^\frac{2q}{6-q}_{q,\beta}(R)\|u\|_{L^6(A_R)}^{\frac{6-2p}{6-p}}  \|v\|_{L^6(A_R)}^{\frac{6-3q}{6-q}}\|\nabla v\|_{L^2(A_R)}\notag\\
&+CR^{\frac{2q}{6-q}\beta-1}Y^\frac{2q}{6-q}_{q,\beta}(R)\|v\|_{L^6(A_R)}^{\frac{6-3q}{6-q}}\|\nabla v\|_{L^2(A_R)}\\
&+CR^{\delta_8}X^\frac{p}{6-p}_{p,\alpha}(R)Y^\frac{6+q}{6-q}_{q,\beta}(R)\|u\|_{L^6(A_R)}^{\frac{6-2p}{6-p}} \|v\|_{L^6(A_R)}^{\frac{6-3q}{6-q}}.
\end{align*}
Letting $R=R_j\rightarrow+\infty$ and using \eqref{ine3.45}, we deduce that $u=v=\nabla\theta=0$. Therefore, $\theta$ is a constant.

{\bf Assume that \eqref{c1.4}  holds.} Then we can choose a sequence $R_j\nearrow+\infty$ such that
\begin{equation*}
\lim\limits_{j\rightarrow+\infty}X_{3,\alpha}(R_j)=0,\;\lim\limits_{j\rightarrow+\infty}Y_{q,\beta}(R_j)<+\infty,\;\lim\limits_{j\rightarrow+\infty}Y'_{2,\gamma}(R_j)<+\infty.
\end{equation*}
The rest of the proof follows almost the same line as the first case. The only difference lies in the treatment of $J_3$,  where we consistently employ \eqref{ine3.17} with $p=3$.
\end{proof}

To prove Corollaries \ref{Cor1.1} and \ref{Cor1.2}, we need the following lemma.

\begin{Lem}\label{Lem3.8}
Let $g$ be a smooth function vanishing at infinity, i.e.,
\begin{align*}
\lim_{|x|\rightarrow+\infty}g(x)=0,
\end{align*}
and $\nabla g\in L^r(\mathbb{R}^3)$, where $1\leq r<3$. Then we have $g\in L^\frac{3r}{3-r}(\mathbb{R}^3)$.
\end{Lem}
\begin{proof}
By L'Hospital's rule and the vanishing condition at infinity, we have
\begin{align}\label{ine3.46}
\lim_{R\rightarrow+\infty}(g)_{B_R}=\lim_{R\rightarrow+\infty}\frac{1}{|\partial B_R|}\int_{\partial B_R}gdS=0.
\end{align}
Let $\chi_{B_R}$ denote the characteristic function of $B_R$. Using \eqref{ine3.46}, Fatou's Lemma and the Poincar\'{e}-Sobolev inequality, we deduce
\begin{align*}
\left\|g\right\|_{L^{\frac{3r}{3-r}}(\mathbb{R}^3)}&=\left\|\liminf_{R\rightarrow+\infty}[g-(g)_{B_R}]\chi_{B_R}\right\|_{L^{\frac{3r}{3-r}}(\mathbb{R}^3)}\\
&\leq\liminf_{R\rightarrow+\infty}\left\|g-(g)_{B_R}\right\|_{L^{\frac{3r}{3-r}}(B_R)}\\
&\leq C\liminf_{R\rightarrow+\infty}\|\nabla g\|_{L^r(B_R)}\\
&=C\|\nabla g\|_{L^r(\mathbb{R}^3)}.
\end{align*}
\end{proof}

Now we are in a position to give the proofs of Corollaries \ref {Cor1.1} and \ref {Cor1.2}.

\begin{proof}[{\bf Proof of Corollary \ref{Cor1.1}}]
Since $v$ vanishes at infinity and $\nabla v\in L^2(\mathbb{R}^3)$, applying Lemma \ref{Lem3.8}, we get $v\in L^6(\mathbb{R}^3)$.
Let $\beta=\frac{3}{q}-\frac{1}{2}$.
By the H\"{o}lder inequality, we have
$$Y_{q,\beta}(R)\leq CR^{\frac{3}{q}-\frac{1}{2}-\beta}\|v\|_{L^6(A_R)}=C\|v\|_{L^6(A_R)}\rightarrow0, \text{ as }R\rightarrow+\infty.$$
Taking $\alpha=0,\beta=\frac{3}{q}-\frac{1}{2}, \gamma=0$ in Theorem \ref{main1}, and considering the above fact, we obtain the desired conclusion.
\end{proof}

\begin{proof}[{\bf Proof of Corollary \ref{Cor1.2}}] For $1<p\leq\frac{3}{2}$, Corollary \ref{Cor1.2} is a direct consequence of Corollary \ref{Cor1.1} and Lemma \ref{Lem3.8}.

When $p=1$, using Lemma \ref{Lem3.8}, we find $u\in L^\frac{3}{2}(\mathbb{R}^3)$. On the other hand, since $u$ is smooth and vanishes at infinity, we have $u\in L^\infty(\mathbb{R}^3)$.
Hence, by the interpolation inequality, we obtain $u\in L^3(\mathbb{R}^3)$. The desired conclusion then follows directly from Corollary \ref{Cor1.1}.

\end{proof}

\section{Proof of Theorem \ref{main2}}\label{sec4}

In this section, let $\eta$ be a cut-off function defined by
\begin{align*}
	\eta(x)= \begin{cases}
		1, & |x| <R, \\
		2-\frac{|x|}{R},& R\leq |x|\leq 2R,\\
		0, & |x| >2R,
	\end{cases}
\end{align*}
and $\Theta:=\theta-(\theta)_{A_R}$. Here we point out that the definitions of $\Theta$ and $\eta$ in the present section are slightly different from those in Section \ref{sec3}.
For any $R>0$, we define the energy function $E(R)$ by
\begin{align}\label{ine4.1}
	\aligned
	E(R)=\int_{\mathbb{R}^3}\left(|\nabla u|^{2}+|\nabla v|^{2}+|\nabla \theta|^{2}\right)\eta(x)d x.
	\endaligned
\end{align}
We will show some properties of $E(R)$ in the next two lemmas. Lemma \ref{Lem4.1} gives a lower bound estimate for the derivative $E'(R)$.
In Lemma \ref{Lem4.2}, we establish an upper bound estimate for $E(R)$.
\begin{Lem}\label{Lem4.1}
	Let $(u,\pi,v,\theta)$ be a smooth solution of \eqref{equ1.1} and $E(R)$ be defined by \eqref{ine4.1}. Then we have
	\begin{align}\label{ine4.2}
		E'(R)\geq\frac{1}{R}\int_{A_R}\left(|\nabla u|^{2}+|\nabla v|^{2}+|\nabla \theta|^{2}\right)dx.
	\end{align}
Consequently, $E(R)$ is a non-decreasing function with respect to $R$.
\end{Lem}
\begin{proof}
	We rewrite $E(R)$ as the following form
	$$
	E(R)=\int_{B_R}\left(|\nabla u|^{2}+|\nabla v|^{2}+|\nabla \theta|^{2}\right) dx+\int_{A_R}\left(|\nabla u|^{2}+|\nabla v|^{2}+|\nabla \theta|^{2}\right)\left(2-\frac{|x|}{R}\right)dx.
	$$
	By a direct calculation, we obtain
	\begin{align*}
		E'(R)=&\int_{\partial B_R}\left(|\nabla u|^{2}+|\nabla v|^{2}+|\nabla \theta|^{2}\right) dS+\int_{A_R}\left(|\nabla u|^{2}+|\nabla v|^{2}+|\nabla \theta|^{2}\right)\frac{|x|}{R^2}dx\\
		&+2\int_{\partial B_{2R}}\left(|\nabla u|^{2}+|\nabla v|^{2}+|\nabla \theta|^{2}\right)\left(2-\frac{2R}{R}\right)dS\\
		&-\int_{\partial B_R}\left(|\nabla u|^{2}+|\nabla v|^{2}+|\nabla \theta|^{2}\right)\left(2-\frac{R}{R}\right)dS\\
		=&\int_{A_R}\left(|\nabla u|^{2}+|\nabla v|^{2}+|\nabla \theta|^{2}\right)\frac{|x|}{R^2}dx\\
		\geq&\frac{1}{R}\int_{A_R}\left(|\nabla u|^{2}+|\nabla v|^{2}+|\nabla \theta|^{2}\right)dx.
	\end{align*}
\end{proof}

\begin{Lem}\label{Lem4.2}
	Let $(u,\pi,v,\theta)$ be a smooth solution of \eqref{equ1.1} and $E(R)$ be defined by \eqref{ine4.1}. Then for any $R>0$, it holds that
	\begin{align}\label{ine4.3}
		E(R)\leq& C\left(\|\nabla u\|_{L^2(A_R)}^2+\|\nabla v\|_{L^2(A_R)}^2+\|\nabla \theta\|_{L^2(A_R)}^2\right)+CR^{\frac{1}{2}-\frac{3}{p}}\|u\|_{L^p(A_R)}\|\nabla u\|_{L^2(A_R)}\notag\\
		&+CR^{\frac{1}{2}-\frac{3}{q}}\|v\|_{L^q(A_R)}\|\nabla v\|_{L^2(A_R)}+CR^{-1}\|u\|_{L^3(A_R)}^3\notag\\
		&+CR^{-1}\|u\|_{L^p(A_R)}\|v\|_{L^{2p'}(A_R)}^2+CR^{-1}\|u\|_{L^p(A_R)}\left\|\Theta\right\|_{L^{2p'}(A_R)}^2\\
        &+CR^{-1}\|v\|_{L^2(A_R)}\left\|\Theta\right\|_{L^2(A_R)}.\notag
	\end{align}
\end{Lem}
\begin{proof}
By Lemma \ref{Lem2.1}, there exists $w\in W_{0}^{1,\sigma}(A_R)$ such that $w$ satisfies the following equation
	\begin{align*}
		\mathrm{div} w=u\cdot\nabla\eta \text{ in }A_R,
	\end{align*}
	with the estimate
	\begin{align}\label{ine4.4}
		\|\nabla w\|_{L^\sigma(A_R)}\leq C\|u\cdot\nabla\eta\|_{L^\sigma(A_R)}\leq CR^{-1}\|u\|_{L^\sigma(A_R)},
	\end{align}
	for any $1<\sigma<+\infty$. We extend $w$ by zero to $B_R$, then $w\in W_{0}^{1,\sigma}(B_{2R}).$
	
Obviously, $(u,\pi,v,\Theta)$ also satisfies \eqref{equ1.1}.	
Multiply both sides of $\eqref{equ1.1}_{1}$, $\eqref{equ1.1}_{2}$ and $\eqref{equ1.1}_{3}$ by $u \eta-w$, $v \eta$ and $\Theta \eta$ respectively, integrate over $B_{2R}$ and apply integration by parts. This procedure yields
	\begin{align}\label{ine4.6}
		E&(R)=\int_{B_{2R}}\left(|\nabla u|^{2}\eta+|\nabla v|^{2}\eta+|\nabla \Theta|^{2}\eta\right) d x\notag\\
		=&-\int_{B_{2R}}\Big[\nabla u:(u\otimes\nabla\eta)+\nabla v:(v\otimes\nabla\eta)+\nabla \Theta \cdot \left(\Theta\nabla \eta\right)\Big]d x+\int_{B_{2R}}\nabla u:\nabla w dx\notag\\
		&+\frac{1}{2} \int_{B_{2R}}|u|^2u \cdot \nabla \eta d x+\frac{1}{2} \int_{B_{2R}}|v|^2u \cdot \nabla \eta d x+\frac{1}{2} \int_{B_{2R}}|\Theta|^2u \cdot \nabla \eta d x\\
		&- \int_{B_{2R}}(u \cdot\nabla )w \cdot u dx+\int_{B_{2R}}(u \cdot v )v \cdot \nabla\eta dx -\int_{B_{2R}}(v \cdot\nabla)w\cdot v dx+\int_{B_{2R}}\Theta v\cdot\nabla \eta dx\notag\\
       =&:\sum_{i=1}^9L_i.\notag
	\end{align}

Similarly to \eqref{ine2.3}, we have
	\begin{align}
		\|u\|_{L^2(A_R)}&\leq CR\|\nabla u\|_{L^2 (A_R)}+CR^{\frac{3}{2}-\frac{3}{p}}\|u\|_{L^p(A_R)},\label{ine4.7}\\
        \|v\|_{L^2(A_R)}&\leq CR\|\nabla v\|_{L^2 (A_R)}+CR^{\frac{3}{2}-\frac{3}{q}}\|v\|_{L^q(A_R)}.\label{ine4.8}
	\end{align}
With the help of the H\"{o}lder inequality, \eqref{ine4.7}, \eqref{ine4.8} and the Poincar\'{e} inequality, we have
	\begin{align}\label{ine4.9}
		L_1\leq& CR^{-1}\left(\|\nabla u\|_{L^2 (A_R)}\|u\|_{L^2 (A_R)}+\|\nabla v\|_{L^2(A_R)}\|v\|_{L^2(A_R)}+\|\nabla \Theta\|_{L^2 (A_R)}\|\Theta\|_{L^2 (A_R)}\right)\notag\\
        \leq& C\left(\|\nabla u\|_{L^2 (A_R)}^2+\|\nabla v\|_{L^2 (A_R)}^2+\|\nabla \Theta\|_{L^2 (A_R)}^2\right)+CR^{\frac{1}{2}-\frac{3}{p}}\|\nabla u\|_{L^2 (A_R)}\|u\|_{L^p(A_R)}\\
        &+CR^{\frac{1}{2}-\frac{3}{q}}\|\nabla v\|_{L^2 (A_R)}\|v\|_{L^q(A_R)}.\notag
	\end{align}
	Using the H\"{o}lder inequality, \eqref{ine4.4} and \eqref{ine4.7}, we get
	\begin{align}\label{ine4.10}
		L_2&\leq\|\nabla u\|_{L^2(A_R)}\|\nabla w\|_{L^2(A_R)}\notag\\
		&\leq \|\nabla u\|_{L^2(A_R)}\cdot CR^{-1}\|u\|_{L^2(A_R)}\\
		&\leq C\|\nabla u\|_{L^2(A_R)}^2+CR^{\frac{1}{2}-\frac{3}{p}}\|\nabla u\|_{L^2 (A_R)}\|u\|_{L^p(A_R)}.\notag
	\end{align}
	Using the H\"{o}lder inequality and \eqref{ine4.4}, we obtain
	\begin{align}\label{ine4.11}
		L_3+L_6\leq& CR^{-1}\|u\|_{L^3(A_R)}^3+\|u\|_{L^3(A_R)}^2\|\nabla w\|_{L^3(A_R)}\leq CR^{-1}\|u\|_{L^3(A_R)}^3,
	\end{align}
and
	\begin{align}\label{ine4.12}
		L_4+L_7+L_8\leq& \left(CR^{-1}\|u\|_{L^p(A_R)}+\|\nabla w\|_{L^{p}(A_R)}\right)\|v\|_{L^{2p'}(A_R)}^2\notag\\
        \leq& CR^{-1}\|u\|_{L^p(A_R)}\|v\|_{L^{2p'}(A_R)}^2.
	\end{align}
By the H\"{o}lder inequality, we have
	\begin{align}\label{ine4.13}
		L_5+L_9\leq& CR^{-1}\|u\|_{L^p(A_R)}\|\Theta\|_{L^{2p'}(A_R)}^2+CR^{-1}\|v\|_{L^2(A_R)}\|\Theta\|_{L^2(A_R)}.
	\end{align}
Combining \eqref{ine4.6}, \eqref{ine4.9}, \eqref{ine4.10}, \eqref{ine4.11}, \eqref{ine4.12} and \eqref{ine4.13}, we conclude that \eqref{ine4.3} holds.
\end{proof}

For the sake of convenience, we denote the seven terms on the right hand side of \eqref{ine4.3} by $K_1$, $K_2$, $\cdots$, $K_7$, respectively, i.e.
	\begin{align*}
		&K_1=C\left(\|\nabla u\|_{L^2(A_R)}^2+\|\nabla v\|_{L^2(A_R)}^2+\|\nabla \theta\|_{L^2(A_R)}^2\right),\\
        &K_2=CR^{\frac{1}{2}-\frac{3}{p}}\|u\|_{L^p(A_R)}\|\nabla u\|_{L^2(A_R)},\; K_3=CR^{\frac{1}{2}-\frac{3}{q}}\|v\|_{L^q(A_R)}\|\nabla v\|_{L^2(A_R)}\\
		&K_4=CR^{-1}\|u\|_{L^3(A_R)}^3,\quad \quad\quad \quad\quad \quad K_5=CR^{-1}\|u\|_{L^p(A_R)}\|v\|_{L^{2p'}(A_R)}^2,\\
		&K_6=CR^{-1}\|u\|_{L^p(A_R)}\|\Theta\|_{L^{2p'}(A_R)}^2,\quad K_7=CR^{-1}\|v\|_{L^2(A_R)}\|\Theta\|_{L^2(A_R)}.
	\end{align*}

In order to control $K_1$, $K_6$ and $K_7$ effectively, we need to estimate $f(R)$, which is defined by \eqref{ine3.1}.

\begin{Lem}\label{Lem4.3}
Let the assumptions be the same as those in {\rm Theorem \ref{main2}}. Then there exist three positive constants $R_1>3$, $\tau_5$ and $C$ such that
	\begin{align}\label{ine4.14}
		f(2R)\leq C(\ln R)^{\tau_5},\;\forall R\geq R_1.
	\end{align}
\end{Lem}
\begin{proof}
No matter whether $\mathrm{(B1)}$ or $\mathrm{(B2)}$ holds, there always exist two positive constants $R_1>3$ and $C$ such that the following three inequalities hold for any $R\geq R_1$:	
\begin{align}
&\|u\|_{L^p\left(A_R\right)}\leq CR^\alpha(\ln R)^\lambda,\quad\|v\|_{L^q\left(A_R\right)}\leq CR^\beta(\ln R)^\mu,\quad\|\nabla v\|_{L^2\left(A_R\right)}\leq CR^\gamma.\label{ine4.15}
\end{align}
We only consider the case $p\in(\frac{3}{2},3)$ as the remaining case $p=3$ can be treated similarly.	
In the proof of Theorem \ref{main1}, we derive the estimate \eqref{ine3.44} for $f\left(\sqrt{2}R\right)$.
With \eqref{ine4.15} at hand, we observe that each term in the right hand of \eqref{ine3.44} can be controlled by a power of $\ln R$.
Therefore, there exist positive constants  $\tau_5$ and $C$ such that
\begin{align*}
f\left(\sqrt{2}R\right)\leq C(\ln R)^{\tau_5}.
\end{align*}
Replacing $R$ with $\sqrt{2}R$ in the above inequality, we arrive at \eqref{ine4.14}.
\end{proof}

With the help of Lemma \ref{Lem4.3}, we are able to control $K_1$.

\begin{Lem}\label{Lem4.4}
Let the assumptions be the same as those in  {\rm Theorem \ref{main2}}. Let $R_1$ and $\tau_5$ be the constants given in {\rm Lemma \ref{Lem4.3}}. Then there exists a positive constant $C$ such that
	\begin{align}\label{ine4.16}
	K_1\leq&C\left[R\ln R E'(R)\right]^\frac{\tau_5}{1+\tau_5},\;\forall R\geq R_1.
	\end{align}
\end{Lem}
\begin{proof}
	Using \eqref{ine4.14} and \eqref{ine4.2}, we have
	\begin{align*}
		K_1=&K_1^\frac{1}{1+\tau_5}\cdot K_1^\frac{\tau_5}{1+\tau_5} \\
        \leq&\left[Cf(2R)\right]^\frac{1}{1+\tau_5}\cdot\left[CR E'(R)\right]^\frac{\tau_5}{1+\tau_5}\\
		\leq&C\left[R\ln R E'(R)\right]^\frac{\tau_5}{1+\tau_5},
	\end{align*}
	where we require $R\geq R_1$.
	\end{proof}

The estimates of the remaining six terms $K_2$, $K_3$, $\cdots$, $K_7$ will be given in the subsequent five lemmas.

\begin{Lem}\label{Lem4.5}
Let the assumptions be the same as those in  {\rm Theorem \ref{main2}}.  Let $R_1$  be the constant given in {\rm Lemma \ref{Lem4.3}}. Then there exists a positive constant $C$ such that
\begin{equation}\label{ine4.17}
K_2+K_3\leq C[R\ln R E'(R)]^\frac{1}{2},\;\forall R\geq R_1.
\end{equation}
\end{Lem}
\begin{proof}
According to Assumptions \ref{a1.1} and \ref{a1.4}(iii), we have
$$
\alpha-\left(\frac{3}{p}-\frac{1}{2}\right)<0,\;\beta-\left(\frac{3}{q}-\frac{1}{2}\right)\leq0,
$$
and
$$
\mu\leq\frac{1}{2}\cdot\frac{3-p}{6-p}<\frac{1}{6}\;\text{ if }\;\beta=\frac{3}{q}-\frac{1}{2}.
$$
Consequently, there exists a constant $C$ such that for any $R>3$, it holds that
$$
R^{\alpha-\left(\frac{3}{p}-\frac{1}{2}\right)}(\ln R)^\lambda\leq C\;\text{ and }\;R^{\beta-\left(\frac{3}{q}-\frac{1}{2}\right)}(\ln R)^\mu\leq C(\ln R)^\frac{1}{6}.
$$
Then using \eqref{ine4.15} and \eqref{ine4.2}, for any $R\geq R_1$, we have
\begin{align*}
K_2+K_3&\leq  CR^{\alpha-\left(\frac{3}{p}-\frac{1}{2}\right)}(\ln R)^\lambda \|\nabla u\|_{L^2(A_R)}+ CR^{\beta-\left(\frac{3}{q}-\frac{1}{2}\right)}(\ln R)^\mu\|\nabla v\|_{L^2(A_R)}\\
&\leq C\|\nabla u\|_{L^2(A_R)}+C(\ln R)^\frac{1}{6}\|\nabla v\|_{L^2(A_R)}\\
&\leq C[R E'(R)]^\frac{1}{2}+C[R\ln R E'(R)]^\frac{1}{2}\\
&\leq C[R\ln R E'(R)]^\frac{1}{2}.
\end{align*}
\end{proof}

\begin{Lem}\label{Lem4.6}
Let the assumptions be the same as those in {\rm Theorem \ref{main2}}. Assume $E(R)\not\equiv0$. Then there exist two positive constants $R_5>3$ and $C$ such that
\begin{align}\label{ine4.18}
K_4\leq \frac{1}{16}E(R)+C\chi(p)\left[R\ln R E'(R)\right]^{\frac{9-3p}{6-p}},\;\forall R\geq R_5,
\end{align}
where $\chi(p)$ is defined by
\begin{equation*}
\chi(p)=
\begin{cases}
1, & \text{ if  }\frac{3}{2}<p<3, \\
0, & \text{ if  } p=3.
\end{cases}
\end{equation*}
\end{Lem}
\begin{proof}
Using the Minkowski inequality, the Poincar\'{e}-Sobolev inequality,  and the H\"{o}lder inequality, we derive
	\begin{align}\label{ine4.19}
		\|u\|_{L^6(A_R)}&\leq \|u-(u)_{A_R}\|_{L^6(A_R)}+\|(u)_{A_R}\|_{L^6(A_R)}\notag\\
		&\leq C\|\nabla u\|_{L^2 (A_R)}+CR^\frac{1}{2}\left|(u)_{A_R}\right|\\
		&\leq C\|\nabla u\|_{L^2 (A_R)}+CR^{\frac{1}{2}-\frac{3}{p}}\|u\|_{L^p(A_R)}.\notag
	\end{align}

When $p\in\left(\frac{3}{2},3\right)$, using the interpolation inequality, \eqref{ine4.19}, \eqref{ine4.15}, \eqref{ine4.2} and the assumptions $\alpha\leq\frac{2}{p}-\frac{1}{3}$, $\lambda\leq\frac{3}{p}-1$, we obtain
	\begin{align}\label{ine4.20}
		K_4&\leq CR^{-1}\|u\|_{L^p(A_R)}^\frac{3p}{6-p}\|u\|_{L^6(A_R)}^\frac{18-6p}{6-p}\notag\\
		&\leq CR^{-1}\|u\|_{L^p(A_R)}^{\frac{3p}{6-p}}\left[\|\nabla u\|_{L^2(A_R)}^\frac{18-6p}{6-p}+\left(R^{\frac{1}{2}-\frac{3}{p}}\|u\|_{L^p(A_R)}\right)^\frac{18-6p}{6-p}\right]\\
        &\leq C[R\ln RE'(R)]^\frac{9-3p}{6-p}+CR^{-\frac{3-p}{p}}(\ln R)^\frac{9-3p}{p},\notag
	\end{align}
for any $R\geq R_1$.
Since $E(R)\not\equiv0$, in view of the non-decreasing property of $E(R)$, there exists a constant $R_2>R_1$ such that
\begin{align}\label{ine4.21}
\text{$E(R)\geq E(R_2)>0$ for any $R\geq R_2$.}
\end{align}
Due to the fact
$$
\lim\limits_{R\rightarrow+\infty}CR^{-\frac{3-p}{p}}(\ln R)^\frac{9-3p}{p}=0,
$$
there exists a constant $R_3>R_2$ such that
\begin{align}\label{ine4.22}
CR^{-\frac{3-p}{p}}(\ln R)^\frac{9-3p}{p}\leq\frac{1}{16}E(R_2)\leq \frac{1}{16}E(R),\;\forall R\geq R_3.
\end{align}
By a combination of \eqref{ine4.20} and \eqref{ine4.22}, we get
\begin{align}\label{ine4.23}
K_4\leq \frac{1}{16}E(R)+C\left[R\ln R E'(R)\right]^{\frac{9-3p}{6-p}},\;\forall R\geq R_3.
\end{align}

When $p=3$, we have
\begin{align*}
\lim\limits_{R\rightarrow+\infty}K_4=\lim\limits_{R\rightarrow+\infty}CR^{3\alpha-1}X_{3,\alpha}^3(R)=0.
\end{align*}
Therefore, there exists a constant $R_4>R_2$ such that
\begin{align}\label{ine4.24}
K_4\leq\frac{1}{16}E(R_2)\leq \frac{1}{16}E(R),\;\forall R\geq R_4.
\end{align}

Denote $R_5=\chi(p)R_3+[1-\chi(p)]R_4$. Taking \eqref{ine4.23} and \eqref{ine4.24} into consideration together, we find that \eqref{ine4.18} holds.
\end{proof}

\begin{Lem}\label{Lem4.7}
Let the assumptions be the same as those in  {\rm Theorem \ref{main2}}. Assume $E(R)\not\equiv0$. Then there exist two positive constants $R_6>3$ and $C$ such that
\begin{align}\label{ine4.25}
K_5\leq \frac{1}{16}E(R)+C \left[ R \ln R E'(R) \right]^{1-\frac{(2p-3)q}{(6-q)p}},\;\forall R\geq R_6.
\end{align}
\end{Lem}
\begin{proof}
From \eqref{ine3.42}, we see that $\delta_3\leq0$ and $\delta_3$ attains zero if and only if
\begin{align*}
\begin{cases}
\frac{9-3p}{6-p}\alpha=\frac{9-3p}{6-p}\left(\frac{2}{p}-\frac{1}{3}\right)\\
\beta=\frac{6-q}{2q}\left(1-\frac{p}{6-p}\alpha\right).
\end{cases}
\end{align*}
That is to say,
\begin{align}\label{ine4.26}
\delta_3\begin{cases}
=0, & \text{ if }\;\frac{3}{2}<p<3\text{ and } (\alpha,\beta)=\left(\frac{2}{p}-\frac{1}{3},\frac{2}{q}-\frac{1}{3}\right), \\
=0, & \text{ if }\;p=3\text{ and } \alpha+\frac{2q}{6-q}\beta=1,\\
<0, & \text{ otherwise}.
\end{cases}
\end{align}

On the other hand, we claim that
\begin{align}\label{ine4.27}
\lambda+\frac{(4p-6)q}{(6-q)p}\mu\leq1-\frac{(2p-3)q}{(6-q)p},\;\text{ if }\;(\alpha,\beta)=\left(\frac{2}{p}-\frac{1}{3},\frac{2}{q}-\frac{1}{3}\right).
\end{align}
\begin{align}\label{ine4.27a}
\lambda+\frac{(4p-6)q}{(6-q)p}\mu\leq1-\frac{(2p-3)q}{(6-q)p},\;\text{ if }\;p=3\text{ and } \alpha+\frac{2q}{6-q}\beta=1.
\end{align}
By Assumptions \ref{a1.3} and \ref{a1.4}(i), we find that \eqref{ine4.27a} holds. Next, we prove \eqref{ine4.27} via case-by-case analysis.

\noindent Case 1: $\frac{6-2p}{6-p} + \frac{6-3q}{6-q} < 1$. Using Assumptions \ref{a1.4}(i) and \ref{a1.3}, we derive that
\begin{align*}
\lambda+\frac{(4p-6)q}{(6-q)p}\mu=&\lambda+\frac{2p-3}{p}\cdot\frac{2q}{6-q}\mu\\
\leq&\lambda+\frac{2p-3}{p}\cdot\left(\frac{3-p}{6-p}+\frac{6-3q}{12-2q}+\frac{1}{2}-\frac{p}{6-p}\lambda\right)\\
=&\frac{9-3p}{6-p}\lambda+\frac{2p-3}{p}\cdot\left(\frac{3-p}{6-p}+\frac{6-3q}{12-2q}+\frac{1}{2}\right)\\
\leq&\frac{9-3p}{6-p}\left(\frac{3}{p}-1\right)+\frac{2p-3}{p}\cdot\left(\frac{3-p}{6-p}+\frac{6-3q}{12-2q}+\frac{1}{2}\right)\\
=&1-\frac{(2p-3)q}{(6-q)p}.
\end{align*}

\noindent Case 2: $\frac{6-2p}{6-p} + \frac{6-3q}{6-q}\geq1$ and $\gamma=0$.  Since $\frac{6-2p}{6-p} + \frac{6-3q}{6-q}\geq1$, we conclude that
$$
p\neq3,\;\;q\leq\frac{18-6p}{9-2p} \;\text{ and }\;6-q\geq\frac{36-6p}{9-2p}.
$$
Using Assumptions \ref{a1.4}(ii), \ref{a1.3} and the above fact, we obtain
\begin{align*}
\lambda+\frac{(4p-6)q}{(6-q)p}\mu=&\lambda+\frac{2p-3}{p}\cdot\frac{2q}{6-q}\mu\\
<&\lambda+\frac{2p-3}{p}\cdot\left(1-\frac{p}{6-p}\lambda\right)\\
=&\frac{9-3p}{6-p}\lambda+\frac{2p-3}{p}\\
\leq&\frac{9-3p}{6-p}\left(\frac{3}{p}-1\right)+\frac{2p-3}{p}\cdot\left(\frac{6}{6-q}-\frac{q}{6-q}\right)\\
\leq&\frac{9-3p}{6-p}\left(\frac{3}{p}-1\right)+\frac{2p-3}{p}\cdot\frac{6(9-2p)}{36-6p}-\frac{(2p-3)q}{(6-q)p}\\
=&1-\frac{(2p-3)q}{(6-q)p}.
\end{align*}

\noindent Case 3: $\frac{6-2p}{6-p} + \frac{6-3q}{6-q}\geq1$ and $\gamma>0$. In view of Assumption \ref{a1.2}(iii), this case cannot happen.
Hence, we complete the proof of \eqref{ine4.27}.

Using \eqref{ine4.26}, \eqref{ine4.27} and \eqref{ine4.27a}, we deduce that
\begin{align}\label{ine4.28}
R^{\delta_3}(\ln R)^{\lambda+\frac{(4p-6)q}{(6-q)p}\mu}\leq C(\ln R)^{1-\frac{(2p-3)q}{(6-q)p}},\;\forall R>3.
\end{align}
In view of \eqref{ine4.26} and Assumption \ref{a1.4}(iii), we conclude that $\left[\beta-\left(\frac{3}{q}-\frac{1}{2}\right)\right]\left(2-\frac{(4p-6)q}{(6-q)p}\right)$ and  $\delta_3$ are nonpositive and can not attain zero at the same time. Hence, it follows that
\begin{align}\label{ine4.29}
\delta_3+\left[\beta-\left(\frac{3}{q}-\frac{1}{2}\right)\right]\left(2-\frac{(4p-6)q}{(6-q)p}\right)<0.
\end{align}

Similarly to \eqref{ine4.19}, we have
\begin{align}\label{ine4.30}
\|v\|_{L^6(A_R)}&\leq C\|\nabla v\|_{L^2 (A_R)}+CR^{\frac{1}{2}-\frac{3}{q}}\|v\|_{L^q(A_R)}.
\end{align}
By the interpolation inequality, \eqref{ine4.30}, \eqref{ine4.15}, \eqref{ine4.2} and \eqref{ine4.28}  we obtain
\begin{align}\label{ine4.31}
	K_5 &\leq C R^{-1} \|u\|_{L^p(A_R)} \|v\|_{L^q(A_R)}^\frac{(4p-6)q}{(6-q)p}\|v\|_{L^6(A_R)}^{2-\frac{(4p-6)q}{(6-q)p}}\notag\\
        &\leq C R^{-1} \|u\|_{L^p(A_R)} \|v\|_{L^q(A_R)}^\frac{(4p-6)q}{(6-q)p}\left[\|\nabla v\|_{L^2 (A_R)}^{2-\frac{(4p-6)q}{(6-q)p}}+\left(R^{\frac{1}{2}-\frac{3}{q}}\|v\|_{L^q(A_R)}\right)^{2-\frac{(4p-6)q}{(6-q)p}}\right]\\
        &\leq CR^{\delta_3}(\ln R)^{\lambda+\frac{(4p-6)q}{(6-q)p}\mu}\left[R E'(R) \right]^{1-\frac{(2p-3)q}{(6-q)p}}+CR^{\delta_3+\left[\beta-\left(\frac{3}{q}-\frac{1}{2}\right)\right]\left(2-\frac{(4p-6)q}{(6-q)p}\right)}(\ln R)^{\lambda+2\mu}\notag\\
        &\leq C \left[ R \ln R E'(R) \right]^{1-\frac{(2p-3)q}{(6-q)p}}+CR^{\delta_3+\left[\beta-\left(\frac{3}{q}-\frac{1}{2}\right)\right]\left(2-\frac{(4p-6)q}{(6-q)p}\right)}(\ln R)^{\lambda+2\mu}.\notag
\end{align}

Since $E(R)\not\equiv0$, \eqref{ine4.21} holds. Thanks to \eqref{ine4.29}, we get
$$
\lim\limits_{R\rightarrow+\infty}CR^{\delta_3+\left[\beta-\left(\frac{3}{q}-\frac{1}{2}\right)\right]\left(2-\frac{(4p-6)q}{(6-q)p}\right)}(\ln R)^{\lambda+2\mu}=0.
$$
Consequently, there exists a constant $R_6>R_2$ such that
\begin{align}\label{ine4.32}
CR^{\delta_3+\left[\beta-\left(\frac{3}{q}-\frac{1}{2}\right)\right]\left(2-\frac{(4p-6)q}{(6-q)p}\right)}(\ln R)^{\lambda+2\mu}\leq\frac{1}{16}E(R_2)\leq \frac{1}{16}E(R),\;\forall R\geq R_6.
\end{align}
Combining \eqref{ine4.31} and \eqref{ine4.32}, we obtain \eqref{ine4.25}.
\end{proof}

\begin{Lem}\label{Lem4.8}
	Let the assumptions be the same as those in  {\rm Theorem \ref{main2}}. Then there exists a positive constant $C$ such that
	\begin{equation}\label{ine4.33}
		K_6 \leq C \left[ R \ln R E'(R) \right]^{\frac{3}{2p}},\;\forall R\geq R_1.
	\end{equation}
\end{Lem}
\begin{proof}
Applying Lemma \ref{Lem2.5} with $L=2$ and $\rho=R$, and using the interpolation inequality, \eqref{ine4.15} and \eqref{ine4.14}, we have
	\begin{align}\label{ine4.34}
	\|\Theta\|_{L^2(A_R)} &\leq C\|u\|_{L^p(A_R)}^{\frac{p}{6-p}} \|u\|_{L^6(A_R)}^{\frac{6-2p}{6-p}} \left( \|\nabla v\|_{L^2(A_R)} + R^{\frac{1}{2}-\frac{3}{q}}\|v\|_{L^q(A_R)} \right) + C\|\nabla v\|_{L^2(A_R)} \notag\\
	&\leq C\|u\|_{L^p(A_R)}^{\frac{p}{6-p}} f^{\frac{3-p}{6-p}}(2R) \left[ f^{\frac{1}{2}} (2R)+ R^{\beta-\left(\frac{3}{q}-\frac{1}{2}\right)}\left(\ln R\right)^\mu\right] + C f^{\frac{1}{2}}(2R) \\
	&\leq C (\ln R)^{\frac{12-3p}{12-2p}\tau_{5}+\mu} \|u\|_{L^p(A_R)}^{\frac{p}{6-p}} + C (\ln R)^{\frac{\tau_{5}}{2}}.\notag
\end{align}
By the interpolation inequality, \eqref{ine4.34}, the Poincar\'{e}-Sobolev inequality, \eqref{ine4.15} and \eqref{ine4.2},  we obtain
\begin{align*}
	K_6 &\leq C R^{-1} \|u\|_{L^p(A_R)} \|\Theta\|_{L^2(A_R)}^{2-\frac{3}{p}} \|\Theta\|_{L^6(A_R)}^{\frac{3}{p}} \\
	&\leq C R^{-1} \|u\|_{L^p(A_R)} \left(\left[ (\ln R)^{\frac{12-3p}{12-2p}\tau_{5}+\mu} \|u\|_{L^p(A_R)}^{\frac{p}{6-p}} \right]^{2-\frac{3}{p}}
	+ (\ln R)^{\frac{\tau_{5}}{2}\left(2-\frac{3}{p}\right)}\right) \|\nabla \theta\|_{L^2(A_R)}
^{\frac{3}{p}} \\
	&\leq \left[ C R^{-\frac{2p-3}{3p}} (\ln R)^{\frac{p+3}{6-p}\lambda+\left(\frac{12-3p}{12-2p}\tau_{5}+\mu\right)\left(2-\frac{3}{p}\right)}
	+ C R^{-\frac{4p-6}{3p}} (\ln R)^{\lambda+\left(1-\frac{3}{2p}\right)\tau_{5}} \right] \|\nabla \theta\|_{L^2(A_R)}^{\frac{3}{p}}\\
&\leq C \|\nabla \theta\|_{L^2(A_R)}^{\frac{3}{p}} \\
	&\leq C \left[ R E'(R) \right]^{\frac{3}{2p}} \\
	&\leq C \left[ R \ln R E'(R) \right]^{\frac{3}{2p}}.
\end{align*}

\end{proof}

\begin{Lem}\label{Lem4.9}
Let the assumptions be the same as those in {\rm Theorem \ref{main2}}. Let $\tau_5$ be  the constant given by  {\rm Lemma \ref{Lem4.3}}. Let $\chi(p)$ be the function defined in {\rm Lemma \ref{Lem4.6}}. Assume $E(R)\not\equiv0$. Then there exist three positive constants $\tau_9\in(0,1)$, $R_{11}>3$ and $C$ such that for any $R\geq R_{11}$, it holds that
\begin{align}\label{ine4.35}
K_7 \leq &\frac{1}{8} E(R)+C\left[R\ln RE'(R)\right]^{\tau_9}+C\left[R\ln RE'(R)\right]^{\frac{3-p}{6-p}+\frac{1}{2}}\notag\\
&+C\chi(p)\left(\left[R\ln RE'(R)\right]^{\frac{3-p}{6-p}+\frac{6-3q}{12-2q}}+\left[R\ln RE'(R)\right]^{\frac{3-p}{6-p}}\right)\\
&+C\left[R\ln RE'(R)\right]^{\frac{6-3q}{12-2q}+\frac{1}{2}}+C[R\ln RE'(R)]^\frac{\tau_5}{1+\tau_5}+C[R\ln RE'(R)]^\frac{1}{2}.\notag
\end{align}
\end{Lem}
\begin{proof}
By Lemma \ref{Lem2.5}, we have
\begin{align}\label{ine4.36}
		K_7 \leq& C R^{-1} \|v\|_{L^2(A_R)}\|u\|_{L^3(A_R)}\|\nabla v\|_{L^2(A_R)}+CR^{-\frac{3}{q}-\frac{1}{2}}\|v\|_{L^2(A_R)}\|u\|_{L^3(A_R)}\|v\|_{L^q(A_R)}\notag\\
        &+C R^{-1} \|v\|_{L^2(A_R)}\|\nabla v\|_{L^2(A_R)}\\
		=&:K_{71} + K_{72} + K_{73}.\notag
\end{align}

We first derive an estimate for $K_{71}$. When $p=3$, we deduce from  Assumptions \ref{a1.3} and \ref{a1.4}(iii) that $\lambda=0$ and $\beta<\frac{3}{q}-\frac{1}{2}$. Using Assumption \ref{a1.4}(i), we conclude
\begin{align*}
&R^{\alpha+\frac{2q}{6-q}\beta-1}(\ln R)^{\frac{2q}{6-q}\mu}\leq C(\ln R)^{\frac{6-3q}{12-2q}+\frac{1}{2}},\\
&R^{\alpha+\frac{2q}{6-q}\beta-1+\frac{6-3q}{6-q}\left[\beta-\left(\frac{3}{q}-\frac{1}{2}\right)\right]}(\ln R)^{\mu}\leq C(\ln R)^\frac{1}{2}.
\end{align*}
By the interpolation inequality, \eqref{ine4.30}, \eqref{ine4.15}, \eqref{ine4.2} and the above two inequalities, we obtain
\begin{align}\label{ine4.36a}
		K_{71} \leq& C R^{-1}\|u\|_{L^3(A_R)}\|v\|_{L^q(A_R)}^{\frac{2q}{6-q}}\|v\|_{L^6(A_R)}^{\frac{6-3q}{6-q}}\|\nabla v\|_{L^2(A_R)}\notag\\
               \leq& C R^{-1}\|u\|_{L^3(A_R)}\|v\|_{L^q(A_R)}^{\frac{2q}{6-q}}\left(\|\nabla v\|_{L^2(A_R)}^{\frac{6-3q}{6-q}}+\left(R^{\frac{1}{2}-\frac{3}{q}} \|v\|_{L^q(A_R)}\right)^{\frac{6-3q}{6-q}}\right)\|\nabla v\|_{L^2(A_R)}\notag\\
               \leq&C R^{\alpha+\frac{2q}{6-q}\beta-1}(\ln R)^{\frac{2q}{6-q}\mu}\left[R E'(R)\right]^{\frac{6-3q}{12-2q}+\frac{1}{2}}\\
               &+C R^{\alpha+\frac{2q}{6-q}\beta-1+\frac{6-3q}{6-q}\left[\beta-\left(\frac{3}{q}-\frac{1}{2}\right)\right]}(\ln R)^{\mu}\left[R E'(R)\right]^{\frac{1}{2}}\notag\\
               \leq&C\left[R\ln RE'(R)\right]^{\frac{6-3q}{12-2q}+\frac{1}{2}}+C\left[R\ln R E'(R)\right]^{\frac{1}{2}}.\notag
\end{align}

When $\frac{3}{2}<p<3$, the estimate for $K_{71}$ becomes more involved. By the interpolation inequality, \eqref{ine4.19} and \eqref{ine4.30}, we obtain
\begin{align}\label{ine4.37}
	K_{71} \leq& C R^{-1} \|u\|_{L^p(A_R)}^{\frac{p}{6-p}}\|v\|_{L^q(A_R)}^{\frac{2q}{6-q}} \|u\|_{L^6(A_R)}^{\frac{6-2p}{6-p}}\|v\|_{L^6(A_R)}^{\frac{6-3q}{6-q}}  \|\nabla v\|_{L^2(A_R)}\notag\\
	\leq &C R^{-1} \|u\|_{L^p(A_R)}^{\frac{p}{6-p}}\|v\|_{L^q(A_R)}^{\frac{2q}{6-q}}\|\nabla u\|_{L^2(A_R)}^{\frac{6-2p}{6-p}}\|\nabla v\|_{L^2(A_R)}^{\frac{6-3q}{6-q}}\|\nabla v\|_{L^2(A_R)}\notag\\
&+C R^{-1} \|u\|_{L^p(A_R)}^{\frac{p}{6-p}}\|v\|_{L^q(A_R)}^{\frac{2q}{6-q}}\|\nabla u\|_{L^2(A_R)}^{\frac{6-2p}{6-p}}\left(R^{\frac{1}{2}-\frac{3}{q}} \|v\|_{L^q(A_R)}\right)^{\frac{6-3q}{6-q}}\|\nabla v\|_{L^2(A_R)}\\
&+C R^{-1} \|u\|_{L^p(A_R)}^{\frac{p}{6-p}}\|v\|_{L^q(A_R)}^{\frac{2q}{6-q}}\left(R^{\frac{1}{2}-\frac{3}{p}} \|u\|_{L^p(A_R)}\right)^{\frac{6-2p}{6-p}}
\|v\|_{L^6(A_R)}^{\frac{6-3q}{6-q}}\|\nabla v\|_{L^2(A_R)}\notag\\
=&:K_{711} + K_{712} + K_{713}.\notag
\end{align}

To begin with, we establish a bound for $K_{711}$ via case-by-case analysis.
When $\frac{6-2p}{6-p} + \frac{6-3q}{6-q} < 1$, using \eqref{ine4.15}, \eqref{ine4.2} and Assumptions \ref{a1.2}(i), \ref{a1.4}(i), we obtain
\begin{align}\label{ine4.38}
	K_{711} \leq &C(R)^{\frac{p}{6-p}\alpha+\frac{2q}{6-q}\beta-1}(\ln R)^{\frac{p}{6-p}\lambda+\frac{2q}{6-q}\mu}\left[R E'(R)\right]^{\tau_6}\leq C\left[R \ln R E'(R)\right]^{\tau_6},
\end{align}
where $\tau_6$ is defined by
$$
\tau_6=\frac{3-p}{6-p}+\frac{6-3q}{12-2q}+\frac{1}{2}.
$$
When $\frac{6-2p}{6-p} + \frac{6-3q}{6-q}\geq1$ and $\gamma>0$, it follows from Assumption \ref{a1.2}(iii) that
$$
\frac{p}{6-p}\alpha+\frac{2q}{6-q}\beta<1.
$$
Using \eqref{ine4.15}, \eqref{ine4.14} and \eqref{ine4.2}, we get
\begin{align}\label{ine4.39}
K_{711} \leq&C R^{-1} \|u\|_{L^p(A_R)}^{\frac{p}{6-p}}\|v\|_{L^q(A_R)}^{\frac{2q}{6-q}}\left[f(2R)\right]^{\frac{3-p}{6-p}+\frac{6-3q}{12-2q}}\|\nabla v\|_{L^2(A_R)}\notag\\
\leq&C R^{\frac{p}{6-p}\alpha+\frac{2q}{6-q}\beta-1}(\ln R)^{\frac{p}{6-p}\lambda+\frac{2q}{6-q}\mu+\left(\frac{3-p}{6-p}+\frac{6-3q}{12-2q}\right)\tau_5} \left[RE'(R)\right]^{\frac{1}{2}}\notag\\
\leq&C\left[RE'(R)\right]^{\frac{1}{2}}\\
\leq&C\left[R\ln RE'(R)\right]^{\frac{1}{2}}.\notag
\end{align}
When $\frac{6-2p}{6-p} + \frac{6-3q}{6-q}\geq1$ and $\gamma=0$, thanks to Assumption \ref{a1.4}(ii) and the fact that $\frac{6-2p}{6-p} + \frac{6-3q}{6-q}\in[1,2)$, we can choose a positive number
$\tau_7$ such that
$$
2\left(\frac{p}{6-p}\lambda+\frac{2q}{6-q}\mu\right)-\frac{6-2p}{6-p}-\frac{6-3q}{6-q}<\tau_7<2-\frac{6-2p}{6-p} -\frac{6-3q}{6-q}.
$$
We denote
$$
\tau_8=\frac{\tau_7}{2}+\frac{1}{2}\left(\frac{6-2p}{6-p} + \frac{6-3q}{6-q}\right).
$$
It is not difficult to verify that
$$
\frac{p}{6-p}\lambda+\frac{2q}{6-q}\mu<\tau_8<1.
$$
Using \eqref{ine4.15} and \eqref{ine4.2}, we conclude that
\begin{align}\label{ine4.40}
K_{711} \leq&C R^{-1} \|u\|_{L^p(A_R)}^{\frac{p}{6-p}}\|v\|_{L^q(A_R)}^{\frac{2q}{6-q}}\|\nabla v\|_{L^2(A_R)}^{1-\tau_7}\|\nabla u\|_{L^2(A_R)}^{\frac{6-2p}{6-p}}\|\nabla v\|_{L^2(A_R)}^{\frac{6-3q}{6-q}}\|\nabla v\|_{L^2(A_R)}^{\tau_7}\notag\\
\leq&C R^{\frac{p}{6-p}\alpha+\frac{2q}{6-q}\beta-1}(\ln R)^{\frac{p}{6-p}\lambda+\frac{2q}{6-q}\mu} \left[RE'(R)\right]^{\tau_8}\\
\leq&C\left[R\ln RE'(R)\right]^{\tau_8}.\notag
\end{align}
For the sake of convenience, we define a new parameter $\tau_9$ as follows:
\begin{equation*}
\tau_9=
\begin{cases}
\tau_6, & \text{ if  }\frac{6-2p}{6-p} + \frac{6-3q}{6-q} < 1, \\
\tau_8, & \text{ if  } \frac{6-2p}{6-p} + \frac{6-3q}{6-q}\geq1\text{ and }\gamma=0,\\
\frac{1}{2}, & \text{ if  } \frac{6-2p}{6-p} + \frac{6-3q}{6-q}\geq1\text{ and }\gamma>0.
\end{cases}
\end{equation*}
Then putting \eqref{ine4.38}, \eqref{ine4.39} and \eqref{ine4.40} together, we obtain
\begin{align}\label{ine4.41}
K_{711}\leq&C\left[R\ln RE'(R)\right]^{\tau_9}.
\end{align}

Next, we deal with the terms $K_{712}$ and $K_{713}$. For $\frac{3}{2}<p<3$ and $R>3$, we claim
 $$
 R^{\frac{p}{6-p}\alpha+\frac{2q}{6-q}\beta-1+\frac{6-3q}{6-q}\left[\beta-\left(\frac{3}{q}-\frac{1}{2}\right)\right]}(\ln R)^{\frac{p}{6-p}\lambda+\mu}\leq C(\ln R)^{\frac{3-p}{6-p}+\frac{1}{2}}.
 $$
Indeed, when $q=2$, using Assumption \ref{a1.4}(i), we have
 $$
 R^{\frac{p}{6-p}\alpha+\frac{2q}{6-q}\beta-1+\frac{6-3q}{6-q}\left[\beta-\left(\frac{3}{q}-\frac{1}{2}\right)\right]}(\ln R)^{\frac{p}{6-p}\lambda+\mu}=CR^{\delta_6}(\ln R)^{\frac{p}{6-p}\lambda+\mu}\leq C(\ln R)^{\frac{3-p}{6-p}+\frac{1}{2}}.
 $$
 When $1\leq q<2$ and $\beta<\frac{3}{q}-\frac{1}{2}$, we find
 $$
 R^{\frac{p}{6-p}\alpha+\frac{2q}{6-q}\beta-1+\frac{6-3q}{6-q}\left[\beta-\left(\frac{3}{q}-\frac{1}{2}\right)\right]}(\ln R)^{\frac{p}{6-p}\lambda+\mu}\leq C\leq C(\ln R)^{\frac{3-p}{6-p}+\frac{1}{2}}.
 $$
 When $1\leq q<2$ and $\beta=\frac{3}{q}-\frac{1}{2}$, using Assumption \ref{a1.4}(iii), we get
 $$
 R^{\frac{p}{6-p}\alpha+\frac{2q}{6-q}\beta-1+\frac{6-3q}{6-q}\left[\beta-\left(\frac{3}{q}-\frac{1}{2}\right)\right]}(\ln R)^{\frac{p}{6-p}\lambda+\mu}\leq C(\ln R)^{\frac{p}{6-p}\lambda+2\mu}\leq C(\ln R)^{\frac{3-p}{6-p}+\frac{1}{2}}.
 $$
Employing \eqref{ine4.15} and \eqref{ine4.2}, we obtain
\begin{align}\label{ine4.42}
	K_{712}\leq&C R^{\frac{p}{6-p}\alpha+\frac{2q}{6-q}\beta-1+\frac{6-3q}{6-q}\left[\beta-\left(\frac{3}{q}-\frac{1}{2}\right)\right]}(\ln R)^{\frac{p}{6-p}\lambda+\mu}\left[RE'(R)\right]^{\frac{3-p}{6-p}+\frac{1}{2}}\notag\\
  \leq&C\left[R\ln RE'(R)\right]^{\frac{3-p}{6-p}+\frac{1}{2}}.
\end{align}
Employing \eqref{ine4.15} and \eqref{ine4.14}, we obtain
\begin{align*}
K_{713} \leq&C R^{-1} \|u\|_{L^p(A_R)}^{\frac{p}{6-p}}\|v\|_{L^q(A_R)}^{\frac{2q}{6-q}}\left(R^{\frac{1}{2}-\frac{3}{p}} \|u\|_{L^p(A_R)}\right)^{\frac{6-2p}{6-p}}\left[f(2R)\right]^{\frac{6-2q}{6-q}}\\
\leq&CR^{\frac{6-2p}{6-p}\left[\alpha-\left(\frac{3}{p}-\frac{1}{2}\right)\right]}(\ln R)^{\lambda+\frac{2q}{6-q}\mu+\frac{6-2q}{6-q}\tau_5},
\end{align*}
which implies $\lim\limits_{R\rightarrow+\infty}K_{713}=0$.
Consequently, there exists a constant $R_7>R_2$ such that
\begin{align}\label{ine4.43}
K_{713}\leq\frac{1}{16}E(R_2)\leq \frac{1}{16}E(R),\;\forall R\geq R_7.
\end{align}
Combining \eqref{ine4.36a}, \eqref{ine4.37}, \eqref{ine4.41}, \eqref{ine4.42} and \eqref{ine4.43}, we have
\begin{align}\label{ine4.44}
K_{71}\leq&\frac{1}{16}E(R)+C\left[R\ln RE'(R)\right]^{\tau_9}+C\left[R\ln RE'(R)\right]^{\frac{3-p}{6-p}+\frac{1}{2}}\notag\\
&+C\left[R\ln RE'(R)\right]^{\frac{6-3q}{12-2q}+\frac{1}{2}}+C\left[R\ln R E'(R)\right]^{\frac{1}{2}},\;\forall R\geq R_7.
\end{align}

We then turn to handle $K_{72}$.
When $\beta<\frac{3}{q}-\frac{1}{2}$, by the interpolation inequality, \eqref{ine4.15} and \eqref{ine4.14}, we obtain
\begin{align*}
	K_{72} \leq& C R^{-\frac{3}{q}-\frac{1}{2}} \|u\|_{L^p(A_R)}^{\frac{p}{6-p}}\|v\|_{L^q(A_R)}^{\frac{2q}{6-q}} \|u\|_{L^6(A_R)}^{\frac{6-2p}{6-p}}\|v\|_{L^6(A_R)}^{\frac{6-3q}{6-q}}  \|v\|_{L^q(A_R)}\\
 \leq&C R^{-\frac{3}{q}-\frac{1}{2}}\|u\|_{L^p(A_R)}^{\frac{p}{6-p}}\|v\|_{L^q(A_R)}^{\frac{2q}{6-q}}\left[f(2R)\right]^{\frac{3-p}{6-p}+\frac{6-3q}{12-2q}}\|v\|_{L^q(A_R)}\\
 \leq&C R^{\frac{p}{6-p}\alpha+\frac{2q}{6-q}\beta-1+\beta-\left(\frac{3}{q}-\frac{1}{2}\right)}(\ln R)^{\frac{p}{6-p}\lambda+\frac{6+q}{6-q}\mu+\left(\frac{3-p}{6-p}+\frac{6-3q}{12-2q}\right)\tau_5}\\
 \leq&C R^{\frac{1}{2}\left[\beta-\left(\frac{3}{q}-\frac{1}{2}\right)\right]},
\end{align*}
which forces $\lim\limits_{R\rightarrow+\infty}K_{72}=0$. Hence, there exists a constant $R_8>R_2$ such that
\begin{align}\label{ine4.45}
K_{72}\leq\frac{1}{16}E(R_2)\leq \frac{1}{16}E(R),\;\forall R\geq R_8.
\end{align}

When $\beta=\frac{3}{q}-\frac{1}{2}$, we adopt a different strategy to handle $K_{72}$. It is noted that when $\beta=\frac{3}{q}-\frac{1}{2}$, we require $p\neq3$.
By the interpolation inequality, \eqref{ine4.19} and \eqref{ine4.30}, we obtain
\begin{align}\label{ine4.46}
	K_{72} \leq& CR^{-\frac{3}{q}-\frac{1}{2}}\|u\|_{L^p(A_R)}^{\frac{p}{6-p}}\|v\|_{L^q(A_R)}^{\frac{2q}{6-q}} \|u\|_{L^6(A_R)}^{\frac{6-2p}{6-p}}\|v\|_{L^6(A_R)}^{\frac{6-3q}{6-q}} \|v\|_{L^q(A_R)}\notag\\
	\leq &CR^{-\frac{3}{q}-\frac{1}{2}}\|u\|_{L^p(A_R)}^{\frac{p}{6-p}}\|v\|_{L^q(A_R)}^{\frac{2q}{6-q}}\|\nabla u\|_{L^2(A_R)}^{\frac{6-2p}{6-p}}\|\nabla v\|_{L^2(A_R)}^{\frac{6-3q}{6-q}}\|v\|_{L^q(A_R)}\notag\\
&+CR^{-\frac{3}{q}-\frac{1}{2}}\|u\|_{L^p(A_R)}^{\frac{p}{6-p}}\|v\|_{L^q(A_R)}^{\frac{2q}{6-q}}\|\nabla u\|_{L^2(A_R)}^{\frac{6-2p}{6-p}}\left(R^{\frac{1}{2}-\frac{3}{q}} \|v\|_{L^q(A_R)}\right)^{\frac{6-3q}{6-q}}\|v\|_{L^q(A_R)}\\
&+CR^{-\frac{3}{q}-\frac{1}{2}}\|u\|_{L^p(A_R)}^{\frac{p}{6-p}}\|v\|_{L^q(A_R)}^{\frac{2q}{6-q}}\left(R^{\frac{1}{2}-\frac{3}{p}} \|u\|_{L^p(A_R)}\right)^{\frac{6-2p}{6-p}} \|v\|_{L^6(A_R)}^{\frac{6-3q}{6-q}}\|v\|_{L^q(A_R)}\notag\\
=&:K_{721} + K_{722} + K_{723}.\notag
\end{align}
Using \eqref{ine4.15}, \eqref{ine4.2} and Assumption \ref{a1.4}(iii), we obtain
\begin{align}\label{ine4.47}
K_{721}\leq&CR^{\frac{p}{6-p}\alpha+\frac{2q}{6-q}\beta-1+\beta-\left(\frac{3}{q}-\frac{1}{2}\right)}(\ln R)^{\frac{p}{6-p}\lambda+\frac{6+q}{6-q}\mu} \left[RE'(R)\right]^{\frac{3-p}{6-p}+\frac{6-3q}{12-2q}}\notag\\
\leq&C(\ln R)^{\frac{p}{6-p}\lambda+2\mu} \left[RE'(R)\right]^{\frac{3-p}{6-p}+\frac{6-3q}{12-2q}}\notag\\
\leq&C(\ln R)^\frac{3-p}{6-p} \left[RE'(R)\right]^{\frac{3-p}{6-p}+\frac{6-3q}{12-2q}}\\
\leq&C\left[R\ln RE'(R)\right]^{\frac{3-p}{6-p}+\frac{6-3q}{12-2q}},\notag
\end{align}
and
\begin{align}\label{ine4.48}
K_{722}\leq&CR^{\frac{p}{6-p}\alpha+\frac{2q}{6-q}\beta-1+\frac{12-4q}{6-q}\left[\beta-\left(\frac{3}{q}-\frac{1}{2}\right)\right]}(\ln R)^{\frac{p}{6-p}\lambda+2\mu} \left[RE'(R)\right]^{\frac{3-p}{6-p}}\notag\\
\leq&C\left[R\ln RE'(R)\right]^{\frac{3-p}{6-p}}.
\end{align}
Using \eqref{ine4.15} and \eqref{ine4.14}, we obtain
\begin{align*}
K_{723}\leq&CR^{\frac{6-2p}{6-p}\left[\alpha-\left(\frac{3}{p}-\frac{1}{2}\right)\right]+\frac{p}{6-p}\alpha+\frac{2q}{6-q}\beta-1+\beta-\left(\frac{3}{q}-\frac{1}{2}\right)}(\ln R)^{\lambda+\frac{6+q}{6-q}\mu} \left[f(2R)\right]^{\frac{6-3q}{12-2q}}\\
\leq&CR^{\frac{6-2p}{6-p}\left[\alpha-\left(\frac{3}{p}-\frac{1}{2}\right)\right]}(\ln R)^{\lambda+\frac{6+q}{6-q}\mu+\frac{6-3q}{12-2q}\tau_5},
\end{align*}
which implies $\lim\limits_{R\rightarrow+\infty}K_{723}=0$.
Thus, there exists a constant $R_9>R_2$ such that
\begin{align}\label{ine4.49}
	K_{723}\leq\frac{1}{16}E(R_2)\leq \frac{1}{16}E(R),\;\forall R\geq R_9.
\end{align}
We define a constant $R_{10}$ by
\begin{equation*}
R_{10}=
\begin{cases}
R_8, & \text{ if  }\;0\leq\beta<\frac{3}{q}-\frac{1}{2}, \\
R_9, & \text{ if  }\;\beta=\frac{3}{q}-\frac{1}{2}.
\end{cases}
\end{equation*}
Collecting \eqref{ine4.45}, \eqref{ine4.46}, \eqref{ine4.47}, \eqref{ine4.48} and \eqref{ine4.49},
we conclude that for any $R\geq R_{10}$, it holds that
\begin{align}\label{ine4.50}
K_{72}\leq \frac{1}{16}E(R)+C\chi(p)\left(\left[R\ln RE'(R)\right]^{\frac{3-p}{6-p}+\frac{6-3q}{12-2q}}+\left[R\ln RE'(R)\right]^{\frac{3-p}{6-p}}\right).
\end{align}

Using \eqref{ine4.8}, \eqref{ine4.16} and \eqref{ine4.17}, we derive
\begin{align}\label{ine4.51}
	K_{73} &\leq C\|\nabla v\|_{L^2 (A_R)}^2+CR^{\frac{1}{2}-\frac{3}{q}}\|v\|_{L^q(A_R)}\|\nabla v\|_{L^2 (A_R)}\notag\\
           &\leq C[R\ln RE'(R)]^\frac{\tau_5}{1+\tau_5}+C[R\ln RE'(R)]^\frac{1}{2}.
\end{align}
We denote $R_{11}=\max\{R_7,R_{10}\}$. Combining \eqref{ine4.44}, \eqref{ine4.50} and \eqref{ine4.51}, we get \eqref{ine4.35}.
\end{proof}

Now we are ready to prove Theorem \ref{main2}.

\begin{proof}[{\bf Proof of Theorem \ref{main2}}]
We first claim that $E(R)\equiv0$, and we prove this by contradiction.
Assume that $E(R)\not\equiv0$. Then \eqref{ine4.21} holds. Using Lemmas \ref{Lem4.2}, \ref{Lem4.4}, \ref{Lem4.5}, \ref{Lem4.6}, \ref{Lem4.7}, \ref{Lem4.8} and \ref{Lem4.9}, we derive
\begin{align}\label{ine4.52}
E(R)\leq &\frac{1}{4} E(R)+C[R\ln RE'(R)]^\frac{\tau_5}{1+\tau_5}+C[R\ln RE'(R)]^\frac{1}{2}+C\left[R\ln R E'(R)\right]^{1-\frac{(2p-3)q}{(6-q)p}}\notag\\
&+C\left[R\ln R E'(R)\right]^\frac{3}{2p}+C\left[R\ln RE'(R)\right]^{\tau_9}+C\left[R\ln RE'(R)\right]^{\frac{3-p}{6-p}+\frac{1}{2}}\notag\\
&+C\left[R\ln RE'(R)\right]^{\frac{6-3q}{12-2q}+\frac{1}{2}}+C\chi(p)\left[R\ln RE'(R)\right]^\frac{9-3p}{6-p}\\
&+C\chi(p)\left[R\ln RE'(R)\right]^{\frac{3-p}{6-p}+\frac{6-3q}{12-2q}}+C\chi(p)\left[R\ln RE'(R)\right]^{\frac{3-p}{6-p}}\notag\\
=&:\frac{1}{4} E(R)+\sum_{i=1}^{10}M_i,\notag
\end{align}
where $R\geq R_0:=\max\{R_5,R_6,R_{11}\}$.

Let $\tau$ be the maximum among the following eight numbers
$$
\frac{\tau_5}{1+\tau_5},\;1-\frac{(2p-3)q}{(6-q)p},\;\frac{3}{2p},\;\tau_9,\;\frac{3-p}{6-p}+\frac{1}{2},\;\frac{6-3q}{12-2q}+\frac{1}{2},\;\frac{9-3p}{6-p},\; \frac{3-p}{6-p}+\frac{6-3q}{12-2q}.
$$
It is not difficult to verify that $\tau\in(0,1)$.
By the Young inequality and \eqref{ine4.21}, we obtain
\begin{align}\label{ine4.53}
	M_i&\leq \frac{1}{40}E(R_2)+C[R\ln R E'(R)]^\tau\notag\\
&\leq \frac{1}{40}E(R)+C[R\ln R E'(R)]^\tau,
\end{align}
where  $i=1,2,\cdots,10$.

Combining \eqref{ine4.52} and \eqref{ine4.53}, we obtain
\begin{align*}
E(R)&\leq \frac{1}{2}E(R)+C[R\ln R E'(R)]^\tau,
\end{align*}
	which implies
	\begin{align*}
		E(R)
		&\leq C\left[R\ln RE'(R)\right]^\tau.
	\end{align*}
	Consequently, it follows that
	\begin{align*}
		\ln\ln R-\ln\ln R_0=\int_{R_0}^{R}\frac{1}{\rho\ln\rho}d\rho\leq\int_{R_0}^{R}\frac{CE'(\rho)}{E^{\frac{1}{\tau}}(\rho)}d\rho\leq CE(R_0)^{1-\frac{1}{\tau}}<+\infty.
	\end{align*}
	Letting $R\rightarrow+\infty$, the above inequality leads to a contradiction. Therefore, $E(R)\equiv0$.
	
	Thanks to the simple inequality
	$$\|\nabla u\|_{L^2(B_R)}^2+\|\nabla v\|_{L^2(B_R)}^2+\|\nabla \theta\|_{L^2(B_R)}^2\leq E(R),$$
	we conclude that $u,v$ are constant vectors and $\theta$ is a constant. Finally, the limit conditions in $\mathrm{(B1)}$ or $\mathrm{(B2)}$ force $u,v$ to be zero.
\end{proof}

\subsection*{Acknowledgements.}
This work was supported by the National Natural Science Foundation of China (Grant No. 12671140), Science Foundation for the Excellent Youth Scholars of Higher Education of Anhui Province (Grant No. 2023AH030073) and Domestic Study and Research Support Program for Young Key Teachers of Higher Education of Anhui Province (Grant No. JNFX2025027).

\subsection*{Data Availability Statement}
No data was used for the research described in the article.

\subsection*{Conflict of Interest Statement}
The author has no relevant financial or non-financial interests to disclose.

\end {document}